\documentclass[11pt]{amsart}

\usepackage{graphicx}
\usepackage{amssymb}

\usepackage{float}

\usepackage[alphabetic,lite]{amsrefs}
 \usepackage{url}
\usepackage{booktabs}
\usepackage{hyperref}
\usepackage[nameinlink,capitalize]{cleveref}
\usepackage{verbatim}
\usepackage[table]{xcolor}
\usepackage{array}

\usepackage{tabularx}
\usepackage{array}
\usepackage{todonotes}
\usepackage[shortlabels]{enumitem}

\usepackage{soul}

\newcommand{\RN}[1]
    {\MakeUppercase{\romannumeral #1}}

\numberwithin{equation}{section}

\newtheorem{theorem}{Theorem}[section]
\newtheorem{lemma}[theorem]{Lemma}
\newtheorem{proposition}[theorem]{Proposition}
\newtheorem{corollary}[theorem]{Corollary}
\newtheorem{conjecture}[theorem]{Conjecture}

\newtheorem{question}[theorem]{Question}
\newtheorem{remark}[theorem]{Remark}
\newtheorem{definition}[theorem]{Definition}

\newtheorem{example}[theorem]{Example}

\newcommand{\Cr}[0]{\textrm{Cr}}

\newcommand{\LL}{\mathcal{L}}

\newcommand{\PP}[0]{\mathbb{P}}

\newcommand{\R}[0]{\mathbb{R}}

\newcommand{\Q}[0]{\mathbb{Q}}

\DeclareMathOperator{\wdim}{wdim}

\DeclareMathOperator{\HH}{H}

\DeclareMathOperator{\Bs}{Bs}

\DeclareMathOperator{\Eff}{Eff}
\DeclareMathOperator{\Nef}{Nef}
\DeclareMathOperator{\Mov}{Mov}
\def\Pic{\operatorname{Pic}}
\DeclareMathOperator{\N}{N}
\DeclareMathOperator{\NE}{NE}

\DeclareMathOperator{\WCD}{WCD}
\DeclareMathOperator{\SBLD}{SBLD}
\DeclareMathOperator{\MCD}{MCD}

\newcommand{\OO}{\mathcal{O}}

\definecolor{Green}{rgb}{0.01, 0.75, 0.24}

\title[Birational geometry via Weyl actions on curves]{Birational geometry of blowups via Weyl chamber decompositions and actions on curves}

\author[Brambilla]{Maria Chiara Brambilla}\address{
Universit\`a Politecnica delle Marche, via Brecce Bianche, I-60131 Ancona, Italia}
\email{m.c.brambilla@univpm.it}

\author[Dumitrescu]{Olivia Dumitrescu}\address{
1. University of North Carolina at Chapel Hill, 
340 Phillips Hall CB 3250 NC 27599-3250 and 2. Institute of Mathematics of the Romanian Academy "Simion Stoilow" IMAR.
}
\email{dolivia@unc.edu}

\author[Postinghel]{Elisa Postinghel}\address{
Dipartimento di Matematica, Universit\`a degli Studi di Trento, via Sommarive 14
I-38123 Povo di Trento (TN)}
\email{elisa.postinghel@unitn.it}

\author[Santana S\'anchez]{Luis Jos\'e Santana S\'anchez}\address{Departamento de Matemáticas, Estad\'istica e I. O., and  IMAULL, Universidad de La Laguna 38200, La Laguna, Tenerife, Espa\~na}\email{lsantans@ull.edu.es}



\begin{document}

\begin{abstract}
We study the birational geometry of $X^n_s$, the blow-up of $\PP^n_\mathbb{C}$ at $s$ points in general position. 
We identify a set of subvarieties, which we call  Weyl $r$-planes,  that belong to an orbit for the action of the Weyl group on $r$-cycles. They satisfy the following  properties:  they appear as stable base locus of divisors; each Weyl $r$-plane is swept out by an $(n-r)$-moving curve class; moreover, if $s\ge   n+3$, for any fixed $r$ all these curve classes belong to the same orbit for the Weyl action. 

For Mori dream spaces of type $X^n_s$, all such orbits are finite and they allow to reinterpret Mukai's description of the Mori chamber decomposition of the effective cone in terms of $(n-r)$-moving  curve classes, unifying previous different approaches.

 If $X^n_s$ is not a Mori dream space, there are infinitely many Weyl $r$-planes. These yields the definition of  the Weyl chamber decomposition of the pseudoeffective cone of divisors. We pose the question as to whether the nef chamber decomposition can be defined (in the negative part of $\overline{\Eff}(X^n_s)$)  and, if  this is the  case, whether it coincides with the Weyl chamber decomposition. 
We conjecture that the answer is affirmative for $X^3_8$ and $X^5_9$. 

\end{abstract}

\maketitle

\tableofcontents

\section{Introduction}
Let $X^n_s$ denote the complex projective space $\PP^n$ blown up at $s$ general points.
The birational geometry of these varieties has attracted the attention of many authors over the last decades. First of all, we recall that a complete classification of those that are Mori dream spaces is known,  see list \eqref{list MDS}, due to \cite{Mukai01} and \cite{CT}, see also \cite{DM1}. For such families, the property of being a  Mori dream space turns out to be equivalent to that of being of Fano type, see \cite{AM}. 
 If $X^n_s$ is a Mori dream space, then the effective cone of divisors is rational polyhedral, and its extremal rays are spanned by finitely many divisors that, for $s> n+1$, form the orbit of an exceptional divisor for the action of the Weyl group, see \cite{CT} for $s\le n+3$ and \cite{CCF} for $X^{4}_8$.
Moreover, the Mori chamber decomposition ($\MCD$) and  the stable base locus decomposition ($\SBLD$) of the effective cone of any Mori dream space $X^n_s$ coincide and are induced by a hyperplane arrangement,  as observed in \cite{BCP,BDPS},
following a thorough analysis of the  base loci of effective divisors carried  out in \cite{BraDumPos3},\cite{BDP-Ciro}.
Moreover, the
 hyperplane arrangement inducing the MCD is preserved by the Weyl action.

We recall that the Weyl group $W^n_s$ acting on the Chow ring $A(X^n_s)$ is the group generated by standard Cremona transformations, see Section \ref{Cremona-transf}.
Throughout the paper, 
we will assume that the Weyl group is not trivial, that is,  we will always assume that $s\geq n+1$.
The action of $W^n_s$ on divisors was studied by Mukai in \cite{Mukai04} (see also \cite{DPR});
the action on curves  was studied by the second author and Miranda in \cite{DM2, DM3}. 
Moreover, in \cite{DM1} the second author and Miranda computed Weyl actions on 2-cycles in $\mathbb{P}^4$ blown up along $8$ points and all the $(-1)$-Weyl lines interpolating them. In particular, they defined and classified {\it Weyl $r$-planes} in $X^4_8$ as the elements of the effective Weyl orbit of the strict transform of a linear subspace $L_I$ in $X^4_8$ under the blow ups of all its linear cycles,  cf. Definition \ref{weyl plane}. 
 The computation of Weyl $r$-planes 
 requires certain techniques of birational geometry that are difficult to generalize to arbitrary dimension, and therefore Weyl orbits are hard to provide via this definition.

At the same time, an alternative interpretation of {\it Weyl cycles}  in $X^n_s$ was provided in \cite{BDP-Ciro},  
see  Definition \ref{Weyl-cycles}.
This definition is easier to use, due to the authors' previous work
that includes a thorough analysis of the stable base locus of any effective divisor on $X^n_s$.
Such techniques were originally developed to compute the dimension of linear systems of hypersurfaces of $\mathbb{P}^n$ with assigned multiple points.

 The first goal of this paper is to 
unify the theory of such cycles for all  Mori dream spaces of type $X^n_s$.
The first main result of our paper is the following complete description.
\begin{theorem}\label{main}
Assume $s\ge n+1$ and let $X=X^n_s$ be a Mori dream space.
Let $W$ be an element of $A_{r}(X^n_s)$, for $1\leq r\leq n-1$. Then the following statements are equivalent:
\begin{enumerate}
\item
$W$ is a Weyl $r$-plane;
\item
$W$ is a Weyl cycle of dimension $r$;
\item 
$W$ is a stable base locus subvariety.
\end{enumerate}

Moreover if $s\le n+3$ the following is also an equivalent statement:
\begin{itemize}
\item[(4)] $W$ is a join or an exceptional divisor.
\end{itemize}
\end{theorem}
{Stable base locus subvarieties are defined in Section \ref{SBL section} and joins in Section \ref{Background}.}

Recall that for $s=n+4$ and $n=3,4$, the equivalence of the first three statements follows from 
\cites{BDP-Ciro,DM1}.
The equivalence between  (3) and (4)  follows because the SBLD and the MCD  coincide  for Mori dream spaces of the form $X^n_s$.
In this paper we complete the proof of Theorem \ref{main} for $s\le n+3$. 
In particular the equivalence between (1) and (4) is proved in Section \ref{sec WP} (see 
Theorem \ref{2 iff 3}) and 
the equivalence between (2) and (4) is proved in Section \ref{sec WC} (see Proposition \ref{joins are weyl cycles} and  Theorem \ref{Weyl is join}).

In order to obtain our result,
the main tool we use  is the 
explicit study of the cones $\mathcal C_k$ of $k$-moving curves, see \cref{def Ck}.
We recall that 
in  \cite{BDPS} we proved that, for Mori dream spaces of the form $X^n_s$, {\it strong duality} holds, that is
the cone of $k$-moving curves is dual to the cone $\mathcal{D}_k$ of divisors whose stable base locus has codimension strictly larger than $k$, for $k\ge 1$, see Definition \ref{Dk cone definition}.

It is known that the Weyl action preserves the effective cone of divisors, see e.g. \cite[Theorem 4]{dumnicki}, while Proposition \ref{mov} states that it also preserves the movable cone of divisors. {On the other hand, the Weyl action does not preserve the divisorial cones $\mathcal D_k$, see for instance Example \ref{eg}.} 
The cones  of $k$-moving curves $\mathcal C_k$ have a better behavior, in the sense that  the effective Weyl orbit of a $k$-moving curve is contained in $\mathcal C_k$. 
Moreover,  we prove that the cone of $k$-moving curves $\mathcal C_k$ is generated by the elements of the effective Weyl orbits on the $k$-moving curve $$(n-k)h-\sum_{|I|=n-k+1}e_i,$$ along with the two Weyl orbits of extremal $0$-moving curves, see Theorems \ref{strong-duality-thm-n+3},  \ref{strong-duality-thm-X^4_8} and
\ref{strong-duality-thm-X^3_7}.

In the second part of the paper, we turn our attention to the case of blow-ups $X^n_s$ that  are not Mori dream spaces. 
The birational geometry of these 
 spaces   is not yet understood and it is currently studied from other points of view, see e.g.\ \cite{SX} for the threefold $X^3_8$. In Section \ref{section seven}, we discuss our point of view, focusing in particular on the cases $X^3_8$ and $X^5_9$, which turn out to have a similar behavior.
First, in Section \ref{section F}, we describe the crucial role played by the anticanonical curve class $F$, introduced in  \cite{DM2, DM3}, see \cref{DEF anticanonical curve}; which allows to subdivide the N\'eron-Severi space into a positive and a negative part, as in $X_s^2$ where $F=-K_{X_s^2}$, see Section \ref{section F}.
Moreover, we study whether $F$ is moving or not, covering all cases except for $X^n_{n+4}$, $n\ge 6$.
In particular, we show that  $F$ is a limit of moving curves in $X^5_9$ in Proposition \ref{mds}, while that  the same is true for $X^3_8$ follows from the work of \cite{SX}.
Notice that $F\in\mathcal{C}_0$, the cone of moving curves,  if and only if the pseudoeffective cone of divisors lies completely in the negative part of the N\'eron-Severi space. 
Furthermore, since the pseudoeffective cones of $X^3_8$ and $X^5_9$ have infinitely many extremal rays, their dual cone $\mathcal{C}_0$ will have infinitely many generators too; we conjecture that the latter fit in three Weyl orbits, see  Conjecture \ref{conj-5-9}.

We recall that for blow-ups $X^n_{s}$ which are not Mori dream spaces, the {Weyl $1$-planes} are infinitely many 
for $s\geq n+5$, and finitely many for $s=n+4$, see \cite{DM2}. We remark that the notion of {Weyl 1-plane} coincides with that of {$(-1)$-Weyl line} using the terminology of \cite{DM2}.
In Section \ref{section infinity}, we prove the following result.
\begin{theorem}\label{infinity}
Let  $X^n_{s}$ be a non-Mori dream space and $2\leq r\leq n-1$. 
Then there are infinitely many Weyl $r$-planes.
\end{theorem}

In  Section \ref{WSBL section}, we study the stable base locus of effective divisors on $X^n_s$ for any $n$ and $s$.
In particular, in Theorem 
\ref{Weyl-SBL} we show that any Weyl $r$-plane is a stable base locus subvariety, that is: the implication (1)$\Rightarrow$(3) of Theorem \ref{main}
holds in general.
Lemma \ref{weyl base locus} also provides a formula for the exact multiplicity of containment of a Weyl $r$-plane in the base locus of an effective divisor. This allows us to define a decomposition of 
the pseudoeffective cone of divisors $\overline{\Eff}(X^n_s)$, which we call \emph{Weyl chamber decomposition}, see
Section \ref{WCD section}. 
We state  Conjecture \ref{WCD-X59}, which claims  that for $X^3_8$ or $X^5_9$ the Weyl and the stable base locus decompositions agree and they give the nef chamber decomposition. In Question \ref{question neg-eff nonmds} we ask if the same conjecture can be extended to the negative part of the pseudoeffective cone of any $X^n_s$.

Finally, in Section \ref{section dimension}, as a second application of our analysis of the base loci,  we propose the  
notion of {\it Weyl expected dimension} of a linear system of divisors in any $X^n_{s}$.
This definition extends the previous ones  of linear expected dimension and of secant expected dimension,  for  divisors on $X^n_s$ given in \cite[Definition 3.2]{BDP-TAMS} and in \cite[Definition 6.1]{BraDumPos3}  respectively,  and  it generalizes the notion of  Weyl expected dimension  formulated for $X^4_8$ in \cite[Definition 2]{BDP-Ciro}. It is known that
the Weyl expected dimension equals the actual dimension of any divisors on any Mori dream space $X^n_s$ (see \cite{BDP-TAMS} and \cite{LPSS}) with the sole exception of the case $X^4_8$ for which the statement remains open, see \cite[Conjecture 1]{BDP-Ciro}. In Question \ref{conj-wdim}, we ask if the same is true for $X_{n+4}^n$.

\medskip

\subsection*{Acknowledgements}
We would like to thank Cinzia Casagrande for pointing out Example \ref{ex-elliptic} to us.

M.~C.~Brambilla's research is partially funded by the 
European Union Next Generation EU, M4C1, CUP E53C24002320006 - 
2022NBN7TL - Applied Algebraic Geometry of Tensors.

O.~Dumitrescu would like to express her gratitude to MPIM Bonn, for their generosity, hospitality and stimulating environment during her long stay. She is equally grateful to IHES and RIMS for their generosity during her time there. Her research is supported by the NSF-FRG DMS 2152130 grant.

E.~Postinghel's research  is partially
 funded by the European Union Next Generation EU, Mission 4, Component 2 - CUP E53D23005400001.
 
M.~C.~Brambilla and E.~Postinghel are members of INdAM-GNSAGA.

\section{Preliminaries}

Let $X=X^n_s$ denote the 
blow-up of $\PP^n$ at $s$ points in general position.

We denote by $A_r(X)$ the group of classes of cycles of dimension $r$ and by 
 $A(X)=\oplus_{r=0}^n A_r(X)$ the Chow ring of $X$. 

The Picard group of $X$ is spanned by the general hyperplane class $H$ and by the exceptional divisors $E_i$. {A divisor of the form
\begin{equation}\label{divisor} 
D=dH-\sum_{i=1}^s m_i E_i
\end{equation}
is said homogeneous if $m_i=m$ for all $i$.}

The anticanonical divisor class is 
\begin{equation}
\label{anticanonical divisor}
-K_X=(n+1)H-\sum_{i=1}^s (n-1)E_i.\end{equation}

Recall that the Dolgachev-Mukai pairing is the bilinear pairing on the Picard group of $X^n_s$ defined by:
\begin{equation}\label{DM}
 \langle H,H\rangle=n-1, \quad \langle E_i, E_j\rangle=-\delta_{i,j}, \quad \langle H, E_i\rangle=0.
\end{equation}
It is known that $X$ is a Mori dream space if and only if $\langle -K_X, -K_X\rangle>0$ (see \cite{DM2} and references within).
More precisely, the list of varieties of type $X^n_s$ which are Mori dream spaces is the following:
\begin{equation}\label{list MDS}
 X^2_s\text{ for }s\le 8; \
  X^3_s\text{ for }s\le 7;\
  X^4_s\text{ for }s\le 8;\
 X^n_s\text{ for }n\ge 5\text{ and }s\le n+3.
\end{equation}

The Dolgachev-Mukai pairing emerges from the Coxeter theory of the Weyl group acting
on the divisor classes. We recall here the main notions and we refer to \cites{Dolgachev,DM2,DM3} for more details.

\subsection{Cremona transformations and the Weyl group}\label{Cremona-transf}
The {\it standard Cremona transformation based on the coordinate points} of $\PP^n$
 is the birational transformation 
defined by the rational map:
$$
\Cr:=(x_0:\dots: x_n) \mapsto (x_0^{-1}:\dots: x_n^{-1})
$$

From now on assume that $s\ge n+1$. Given any subset $\Gamma\subseteq\{1,\ldots,s\}$ of cardinality $n+1$, we denote by $\Cr_\Gamma$ and call {\it standard Cremona transformation} the map obtained by precomposing $\Cr$ with a projective transformation which takes the points indexed by $\Gamma$ to the coordinate points of $\PP^n$.
A standard Cremona transformation induces an automorphism of $\Pic(X_s^n)$, denoted again by $\Cr_\Gamma$ by abuse of notation,
by sending a divisor \eqref{divisor}
 to  
\begin{equation}\label{cremona-rule}
\Cr_\Gamma(D)= (d-b_\Gamma)H-\sum_{i\in \Gamma}(m_i-b_\Gamma)E_i-\sum_{i\not\in \Gamma} m_i E_i,
\end{equation}
where 
$$b_\Gamma(D)=b_\Gamma := \sum_{i\in \Gamma} m_i-(n-1)d.$$ 
We say that a divisor $D$ \eqref{divisor} is {\it Cremona reduced} if $b_\Gamma(D)\le0$ for any $\Gamma\subseteq\{1,\ldots,s\}$ of cardinality $n+1$.

The {\it Weyl group} $W^n_{s}$  acting on the Chow ring $A(X_s^n)$ is the group generated by the standard Cremona transformations.

The anticanonical divisor $-K_X$ is the only homogeneous divisor, up to multiples,  which is invariant under the action of the Weyl group (if $s\ge n+2$ the assumption homogeneous is not necessary).

The Dolgachev-Mukai pairing \eqref{DM} is also preserved by the Weyl action (see \cite{DPR}).

Following  \cite{BDP-Ciro}, we call {\it Weyl divisor} an element of the  Weyl orbit of an exceptional divisor $E_i$.

\subsection{Weyl actions on curves}\label{Cremona-transf-curves}
Let
$\mathrm{N}_1(X)_{\mathbb R}=(A_1(X)/\equiv)\otimes {\mathbb R}$ be the space of  numerical equivalence classes of $1$-cycles, which is dual to the N\'eron-Severi space $\mathrm{N}^1(X)_{\mathbb R}$ with respect to the standard intersection pairing.
A basis of $\N_1(X_s^n)_\mathbb R$ is given by $h=H^{n-1}$, the general line class in $X^n_s$, and, for $i=1\dots,s$, by $e_i$, the class of a general line inside the exceptional divisor $E_i$.
 Thus,  the intersection product of a divisor $D \in \N^1(X_s^n)_\mathbb R$ of the form  \eqref{divisor} and a 
 1-cycle
 $c$ in $\N_1(X_s^n)_\mathbb R$ of class 
 \begin{equation}\label{curve on X^n_s}
	c=\delta h - \sum_{i=1}^s\mu_i e_i, 
\end{equation}
is computed  by
\begin{equation} \label{n+3: int pair} D\cdot c= d\delta  - \sum_{i=1}^s m_i \mu_i. \end{equation}

The Weyl action on curves (effective $1$-cycle)has been investigated in \cite{DM2}.
For an index set  $\Gamma \subset \{1, \ldots, s\}$  of length  $n+1$, the Cremona transformation $\Cr_\Gamma$ takes a curve class of the form \eqref{curve on X^n_s}
of $\N_1(X_s^n)_\mathbb R$,
 which is not contained in the indeterminacy locus of $\Cr_\Gamma$,
 to the following curve class 
\begin{equation} \label{Cremona curves}
	{\rm Cr}_\Gamma (c) = (\delta - (n-1)a_\Gamma) h - \sum_{i \in \Gamma}(\mu_i - a_\Gamma)e_i - \sum_{i \notin \Gamma}\mu_i e_i.
\end{equation}
where 
$$a_\Gamma(c)=a_\Gamma := \sum_{i \in \Gamma} \mu_i - \delta.$$

We point out that we can formally extend formula \eqref{Cremona curves} to any $1$-cycle; however
 in this case curves can be sent to non-effective cycles, and vice versa.
For instance if $c=h-e_1-e_2$ is
the class of a line in $\PP^n$ through two points $p_1$ and $p_2$,  and $\Gamma=\{1,\ldots,n+1\}$, then $\Cr_\Gamma(c)$ is the following non-effective $1$-cycle class: $-(n-2)h+\sum_{i=3}^{n+1}e_i$.

We recall from \cites{DM2, DM3} the following definition:
\begin{definition}\label{DEF anticanonical curve} The anticanonical curve class for $X^n_s$   is the  unique (up to multiples) symmetric 
curve class invariant under the Weyl group action
\begin{equation}\label{anticanonical curve}
F:=(n+1)h-\sum_{i=1}^{s} e_i.
\end{equation}
\end{definition}
We recall from \cite{DM2, DM3} that $X^n_s$ is a Mori dream space if and only if
$$-K_{X_s^n}\cdot F>0.$$

\subsection{Weyl action on $r$-cycles}

Given an effective cycle class $V\in A_r(X)$, we say that a cycle $V'\in A_r(X)$ is in the {\it effective Weyl orbit} of $V$ if  there is a finite sequence of standard Cremona transformations $\phi=\phi_k\circ\cdots\circ \phi_1$ with $\phi(V)=V'$ and such that $\phi_l\circ\cdots\circ\phi_1(V)$
is an effective $r$-cycle for every $1\le l \le k$. We denote by $(W_s^n\cdot V)^+$ the  effective Weyl orbit of $V$.

We observe that the Weyl group actions on divisors \eqref{cremona-rule} and curves \eqref{Cremona curves} are linear with respect to degrees and multiplicities of the divisor/curve and easy to apply. On the other hand, this is not the case for cycles of arbitrary dimensions $r$, where the general formulas need to involve binomial coefficients in $r$ and $n$.

Hence, in recent works,  different approaches have been adopted in order to study  the relevant families of rigid $r$-cycles with respect to the action of the Weyl group. In particular the families of Weyl $r$-planes and of Weyl cycles have been defined.

We recall first from \cite[Definition 1]{DM1} the following notion:

\begin{definition}
\label{weyl plane}
    A \textbf{Weyl $r$-plane} is an element of the  effective Weyl orbit of the strict transform of an $r$-dimensional plane passing through $r+1$ fixed points in  $A_{r}(X^n_s)$, for $1\leq r\leq n-1$. 
\end{definition}

On the other hand, in \cite[Definition 1]{BDP-Ciro}, a Weyl cycle of dimension $r<n-1$ in $A_{r}(X^n_s)$ is defined to be an irreducible component of the intersection of pairwise orthogonal Weyl divisors. 
We give here an equivalent uniform definition which includes also Weyl divisors as follows.
\begin{definition}
\label{Weyl-cycles}
A {\bf Weyl cycle of dimension $r$} in $A_{r}(X^n_s)$ is an irreducible component of the intersection of  Weyl divisors 
    $\{D_i\}$ such that $\langle D_i, D_j\rangle\leq 0$ with respect to the pairing \eqref{DM}, for any $i,j$.
\end{definition}

Now we show that  Definition \ref{Weyl-cycles} is equivalent to \cite[Definition 1]{BDP-Ciro}.

\begin{proposition}\label{weyl}
 $W$ is a Weyl cycle of dimension $r$ if and only if
\begin{enumerate}
\item  $r=n-1$ and  $W$ is a Weyl divisor, i.e.\ an element of the Weyl orbit of an exceptional divisor,
\item $r< n-1$ and $W$ is an irreducible component of the intersection of pairwise orthogonal Weyl divisors.
\end{enumerate}
\end{proposition}

\begin{proof} 
If $W$ is a divisor $D$, there is nothing to prove, because any Weyl divisor satisfies $\langle D,D\rangle=-1$.

Assume $r\le n-2$ and let $W$ be as in Definition \ref{Weyl-cycles}. If $i=j$, we have $\langle D_i,D_j\rangle=-1$, while if $i\ne j$, then the irreducibility of $D_i$ implies that $\langle D_i, D_j\rangle\geq 0$ (see \cite{DPR}), therefore distinct Weyl divisors $D_i$ and $D_j$ are pairwise orthogonal. i.e. $\langle D_i, D_j\rangle=0$. This proves that $W$ satisfies \cite[Definition 1]{BDP-Ciro}. The converse is easy and left to the reader.
\end{proof}

We show now that the pairwise orthogonal Weyl divisors  span a face of the effective cone of $X^n_s$.

\begin{proposition}\label{eff face}
A set  of Weyl divisors that are pairwise orthogonal with respect to the bilinear pairing \eqref{DM} on $X^n_s$ span a face of the pseudoeffective cone of $X^n_s$.
\end{proposition}
\begin{proof}
It is enough to show that the cone spanned by $D_1,D_2$ is a 2-dimensional face of the effective cone, for any pair of orthogonal  Weyl divisors $D_1$ and $D_2$. Using \cite[Proposition 2.3]{BDP-TAMS},
the multiplicity of containment of $D_i$ in the base locus of an effective divisor $D$ is  $\max\{0,-\langle D, D_i\rangle\}$. Hence, using the orthogonality assumptions,
we have that the base locus of the divisor $a_1 D_1+a_2 D_2$ contains $D_1$ with multiplicity $a_1$ and $D_2$ with multiplicity $a_2$. 
So we  obtain that $|a_1 D_1+a_2 D_2| =\{a_1 D_1 + a_2 D_2\}$, for any $a_1,a_2\ge 0$, and so  $h^0(X,a_1 D_1+a_2 D_2)=1$. This proves that there is no big divisor linearly equivalent to a positive linear combination of $D_1$ and $D_2$.
\end{proof}

\subsection{Stable base locus varieties}\label{SBL section}
For any $X$ with $h^1(\mathcal{O}_X)=0$, the pseudoeffective cone $\overline{\Eff}(X)$ admits a wall-and-chamber decomposition called {\it stable base locus decomposition}, see e.g. \cite[Sect 4.1.3]{Huizenga}. 
Recall that the stable base locus of an effective divisor $D$ is the Zariski closed set
$$\mathbb{B}(D) = \bigcap_{m>0} \Bs(mD)$$
where $\Bs(D)$ is the base locus of the linear system $|D|$.
\begin{example}\label{ex-elliptic} 
In $X=X_8^4$  consider the divisor $-K_X$. In \cite[Corollary 7.6]{CCF} the authors show that 
the transform $R$ of
a smooth rational quintic curve in $\PP^4$ through the $8$ blown-up points is in the base locus of $-K_X$. On the other hand $R$ is not in the stable base locus of $-K_X$, see \cite{Xie}. 
\end{example}

All the divisors in an open chamber  of the stable base locus decomposition have the same stable base locus.
The nef cone $\Nef(X)$ is one of the closed chambers  of the stable base locus decomposition, since it is the closure of the cone generated by divisors with empty stable base locus.

We say that an irreducible and reduced subvariety $Y$ of $X$ is a {\it stable base locus subvariety of $X$} if it is a component of the stable base locus of every effective 
divisor $D$ 
containing $Y$ in its base locus.

A divisor $D$ in $X$ is called {\it movable} if its stable base locus has codimension at least two in $X$. 
The movable cone of $X$ is the convex cone $\Mov(X)$ generated by classes of movable divisors.
We recall from \cite{Payne, BDPS} the following definition.

\begin{definition}\label{Dk cone definition}
	Let $0\le k \le n-1$.
	We define $D_k$ to be the cone generated by the classes of {effective} divisors with no component of the stable base locus of dimension $\ge n-k$;  we denote with
	$\mathcal D_k$ the closure of $D_k$ in $\overline{\Eff}(X)$.
\end{definition}

We have the following  stratification
\begin{equation}\label{filtration-cones-divisors}
 \Nef(X)=\mathcal{D}_{n-1} \subseteq \mathcal{D}_{n-2} \subseteq \cdots \subseteq \mathcal{D}_{1}=\overline{\Mov(X)} \subseteq \mathcal{D}_{0}=\overline{\Eff}(X),
\end{equation}
which is compatible with the stable base locus decomposition of $\overline{\Eff}(X)$ because
 each cone $\mathcal D_k$ is naturally a union of  closed stable base locus chambers. 
We refer to \cite{BDPS} for more details on the properties of such cones.

Note that 
the number of global sections is invariant  for the action of the Weyl group $W^n_s$, see \cite[Theorem 4]{dumnicki}, hence the effective cone $\overline{\Eff}(X^n_s)$ is preserved by such action. 
Furthermore, the cone of movable divisors is invariant under the Weyl action, as proved in the following proposition.

\begin{proposition}\label{mov}
  Let $X^n_s$ be a Mori dream space. Then the movable cone of divisors is preserved under the Weyl action.
\end{proposition}
\begin{proof}
Indeed, \cite[Theorem 4.7]{pragmatic} for $X^n_{s+3}$ and \cite{CCF} 
for $X^4_8$ (for more details see also Section \ref{gale section} below)
 implies that an effective divisor $D$ is movable if and only if  $\langle D, D_i\rangle\geq 0$ for every  Weyl divisor $D_i$.
We conclude by recalling that the Weyl action preserves the Dolgachev-Mukai pairing and  the effective cone $\overline{\Eff}(X^n_s)$. 
\end{proof}

{On the other hand,} the stable base locus is not an invariant, nor does $W^n_s$ preserve the cones $\mathcal{D}_k$, for $k\ge2$.
Indeed, applying a Cremona transformation to a divisor might cause loss or acquisition of {stable} base locus, due to the indeterminacy locus of the transformation.
The following simple example illustrates this fact.
\begin{example}\label{eg}
In $X=X_5^4$  consider the divisor $D=2H-E_1-E_2$.  It is easy to check that  $D$ is nef, i.e.\ $D\in \mathcal{D}_3$. Applying a Cremona transformation we obtain $D'=6H-5E_1-5E_2 - 4E_3 -4E_4 -4E_5$ which is a divisor in $\mathcal{D}_1\setminus \mathcal{D}_2$, because its stable base locus contains planes. 
\end{example}

\subsection{Cones of curves in blow ups of $\PP^n$}\label{section-strongduality-P^n}
We now recall from \cites{Payne, BDPS}
 the notion of $k$-moving curve.

 \begin{definition}\label{def Ck}
Given $0\le k\le n-1$, an irreducible curve $C\subset X$ is called {\em $k$-moving} if it belongs to an algebraic family of curves, whose  irreducible elements cover a Zariski open subset of an effective cycle of dimension (at least) $n-k$. Let 
  	$C_k$  be the cone generated by the classes of $k$-moving curves in  $\N_1(X)_\R$.
  	Let $\mathcal{C}_k=\overline C_k$ be its closure.
  \end{definition}
  
We have the following filtration
  \begin{equation}\label{filtration-cones-curves}
  	\overline{\NE(X)} = \mathcal{C}_{n-1} \supseteq \mathcal{C}_{n-2}\supseteq \cdots \supseteq \mathcal{C}_1\supseteq \mathcal{C}_0.
  \end{equation}
  where $\mathcal C_{n-1}=\overline{\NE(X)}$ is the Mori cone of curves, which is dual to $\Nef(X)$ with respect to the standard intersection pairing.
Moreover, the closed cone $\mathcal C_0$ is dual to the pseudoeffective cone, see \cite[Theorem 0.2]{BDPP}.

In general, it is not true that the cones of $k$-moving curves are dual to the cones of divisors  $\mathcal D_k$. 
 Anyway for the varieties $X^n_s$ which are Mori dream spaces, i.e. those listed in \eqref{list MDS}, we know from \cite[Theorem 1.2]{BDPS} that the strong duality 
 $$\mathcal{C}_k=\mathcal D_k^\vee$$ holds for any $k$.

Recall, from \cite{DM2} the following definitions.

\begin{definition}\label{(i) curves}
Let $X$ be a normal projective variety $X$.  A $(i)$-curve in $X$  is a smooth irreducible curve with normal bundle $\mathcal{O}_X(i)^{n-1}$, for $i\in\{-1,0,1\}$.
\end{definition}

\begin{definition}\label{(i) weyl lines}
In $X^n_s$ we call
\begin{enumerate}
\item  {\it $(-1)$-Weyl line} (class)  an element of the effective Weyl  orbit of $h-e_i-e_j$,
\item  {\it $(0)$-Weyl line} (class)  an element of the  Weyl  orbit of $h-e_i$ 
\item  {\it $(1)$-Weyl line} (class)  an element of the  Weyl  orbit of $h$.
\item  {\it exceptional Weyl line} (class)  an element in the  Weyl  orbit of $e_i$
\end{enumerate}
\end{definition}

Notice that if $n=2$, the notions of $(-1)$-Weyl lines and of exceptional Weyl lines coincide.
 Weyl $1$-planes (see Definition \ref{weyl plane}) coincide with  $(-1)$-Weyl lines. 
 In \cite{DM2}), it is proved that the $(i)$-Weyl lines are examples of $(i)$-curves in $X^n_s$  for $i\in\{-1,0,1\}$. 
Moreover,  we have that
the $(i)$-Weyl lines are $0$-moving curves for $i=0$ and $i=1$, the exceptional Weyl lines
are $1$-moving curves, and 
the $(-1)$-Weyl lines are in $\mathcal{C}_{n-1}\setminus\mathcal{C}_{n-2}$.

 We now describe the cone of $0$-moving curves $\mathcal{C}_{0}$.
 Notice that if $c$ is $0$-moving, then any element of its orbit is also $0$-moving, because the maps in the Weyl group $W^n_s$ are birational. The next lemma proves that also the property to be a generator or an extremal ray of $\mathcal C_0$ is preserved by the Weyl action.
\begin{lemma}
	If $c$ spans an extremal ray of the cone of movable curves $\mathcal{C}_{0}$ on $X^n_s$, then so does any element of the Weyl orbit 
 $W_s^n\cdot c$ .
\end{lemma}

\begin{proof}
 Let $c$  be a curve spanning an extremal ray of $\mathcal{C}_0$. It is enough to prove that $\phi(c)$ is an extremal ray of $\mathcal{C}_{0}$ for a standard Cremona transformation $\phi$.
Assume by contradiction that  $\phi(c)$ does not span an extremal ray. Then there exist curves $\alpha_i$, that are extremal on $\mathcal{C}_0$, and real numbers $a_i$ so that 
$$\phi(c)=\sum_{i=1}^{t} a_i\cdot \alpha_i$$
Since all curves $\alpha_1, \ldots, \alpha_t$ are $0$-moving curves, then $\phi$ can not contract their classes. Applying the involution $\phi$ we obtain
$$c=\sum_{i=1}^{t} a_i\cdot\phi(\alpha_i)$$
Since the Weyl action preserves the effective cone of divisors, then it also preserves its dual, which  is $\mathcal{C}_0$. This implies that $\phi(\alpha_i)$ are $0$-moving curves and, since $a_i>0$, then $c$ is not extremal, which is a contradiction.
\end{proof}

\begin{remark}\label{rmk35}Obviously, the curve classes $h$ and $h-e_i$ span extremal rays of $\mathcal{C}_{0}$ because they are indecomposable as sums of $0$-moving curves.
Hence, by the previous lemma we have that the elements of the Weyl orbits $W^n_s\cdot h$ and $W^n_s\cdot (h-e_i)$ are extremal rays of $\mathcal C_0$ for any $X^n_s$.
 Moreover, 
if $X^n_s$ is a Mori dream space, we known that there are no other extremal rays, because, by \cite[Theorem 6.7]{DM2} for $X^n_{n+3}$,
and by \cite[Section 6.11]{CCF} for $X^4_8$,
the extremal rays of $\mathcal C_0$ 
are the $(0)$-Weyl lines and the $(1)$-Weyl lines. 
\end{remark}

\subsection{Joins} \label{Background}
We conclude this preliminary section recalling definitions and notation  introduced in \cite{BraDumPos3}.

Let $X=X_{s}^n$ be the
blow up of $\PP^n$ at $s\le n+3$ general points.

For any subset $I\subset\{1,\ldots,s\}$ let $L_I$ be the strict transform of the linear span of the points indexed by $I$. For $s=n+3$, let $C$ be the strict transform of the unique rational normal curve of degree $n$ passing through all points while $\sigma_t$ denotes the strict transform of its $t$-secant variety. We also set $\sigma_1:=C$.
Finally, for $t=0$ and $2\le |I|\le n$, or $t>0$ and $|I|\ge0$,
\begin{equation}
\label{join def}
J(L_I,\sigma_t)\end{equation}
denotes the strict transform of the join in $X_{s}^n$ of the linear span of the points indexed by $I$ and the $t$-secant variety of the rational normal curve. 
We have
$$\dim(J(L_I,\sigma_t))=2t+|I|-1.$$
In this notation, we use the following conventions:
$J(L_I,\sigma_t)=L_I$ if $t=0$, $J(L_I,\sigma_t)=\sigma_t$ if $|I|=0$, and  $J(L_I,\sigma_t)=\emptyset$ if $t=|I|=0$.

For an effective divisor of the form \eqref{divisor} 
the exact multiplicity of containment (see \cite[Lemma 4.1]{BraDumPos3}) of  the join $J(L_I,\sigma_t)$ in the base locus of $D$ is
$\max\{0,\kappa_{J(L_I,\sigma_t)}(D)\}$, where
\begin{equation} \label{k_J}
	\kappa_{J(L_I,\sigma_t)}(D) = - (|I| + (n+1)t-1)d + \sum_{i\in I} (t+1)m_i + \sum_{i \notin I} t m_i.
\end{equation}

\section{Weyl $r$-planes}\label{sec WP}

In this section we show that in $X_s^n$ for $n+1\le s \le n+3$, the set of Weyl $r$-planes coincides with the set including the exceptional divisors and the joins $J(L_I,\sigma_t)$, introduced in \eqref{join def}. 
In addition, we obtain a description of the cones of curves $\mathcal C_k$, for every $0\le k \le n-1$, in terms of effective Weyl orbits.

We recall here that in \cite[Theorem 5.5]{BDPS}, the authors proved that the cone $\mathcal D_k$  is dual to the cone of curves $\mathcal C_k$.
The proof is based on the following result, that we include here for the sake of completeness (see proof of \cite[Theorem 5.5]{BDPS}):
\begin{lemma}\label{who-sweeps-what}
For every integer $0\le t \le \frac{n}{2}$ and every index set $I\subset \{1, \ldots, s\}$ of length $0\le |I|\le n-2t$, the join $J(L_I,\sigma_t)$ is swept out by irreducible curves of class
\begin{equation} \label{c: join} c_{I,t}:=  (|I| + (n+1)t-1)h - \sum_{i\in I} (t+1)e_i - \sum_{i \notin I} t e_i. \end{equation}
\end{lemma}

Notice that when $t=0$, the irreducible curves of class $c_{I,0}=(|I|-1)h-\sum_{i\in I} e_i$ sweep out the $(|I|-1)$-plane spanned by the points indexed by $I$. 

In the following discussion we show that the curve classes in \eqref{c: join} form effective Weyl orbits of curves of type $c_{\mathcal I}$, for a suitable $\mathcal I \subset \{1, \ldots, s\}$. This will allow us to conclude that the joins are the only Weyl $r$-plane.

\begin{proposition}\label{proposition curve orbit} Let $X=X_{n+3}^n$. For a fixed $r\in \{1, \ldots, n-2\}$, the set
\[
\left \{   c_{I,t} : |I|+2t-1=r \right\}
\]
is an effective Weyl orbit in $A_1(X)$. 

For $r=n-1$, the set 
\[
\left \{   c_{I,t} : |I|+2t-1=n-1 \right\} \cup \{e_1, \ldots, e_s\}
\]	
is an effective Weyl orbit in $A_1(X)$. 
\end{proposition}

\begin{proof}
	It is easy to compute the orbit of the curves $c_{I,t}$ under the Weyl group action.
It is proved in \cite{DM2} that the Coxeter group for  Mori dream spaces of type $X^n_s$ is finite, and therefore the orbits of the Weyl group on curve classes are finite. We claim that $c_{I,t}$ are all the elements of the effective Weyl orbit $(W^n_s\cdot c_{\mathcal I,0})^+$
 of the curve 
\begin{equation} \label{Cremona zero}
 c_\mathcal I= 
 (|\mathcal I|-1)h-\sum_{i\in \mathcal I} e_i,
 \end{equation} 
where $\mathcal I \subset \{1, \ldots, n+3\}$ is {any} set  
of length $|\mathcal I|= r +1$.  
i.e.
\begin{equation}\label{proof: eff weyl orbit} 
(W^n_s\cdot c_{\mathcal I})^+= \left \{   c_{I,t} : |I|+2t-1=r \right\}.
\end{equation}

Set $S:=\{1,\dots,n+3\}$. To show the claim, we consider the curve class $c_{I,t}$ and a set $\Gamma\subset S$ with   $|\Gamma|=n+1$ and we analyse all possible images of  $c_{I,t}$ via the standard Cremona transformation ${\rm Cr}_{\Gamma}$.
Since the number of points is $s=n+3$, we need to consider three distinct cases:

Case 1. $I\subset \Gamma\subset S$, 
so that $S=\Gamma \cup \{i,j\}$, with $j\ne i$. Set $J:=I\cup \{i,j\}$. Then, applying a standard Cremona transformation based on the points of the set $\Gamma$, we obtain the following curve:
$${\rm Cr}_{\Gamma}(c_{I,t})=c_{J,t-1}.$$
Case 2. Let $I^c$ denote the complement of the set $I$ in $S$, and let  $\Gamma$ satisfy $I^c\subset \Gamma\subset S$, so that $S=\Gamma \cup \{i,j\}$.  Set $J:=I\backslash \{i,j\}$. It is easy to check that applying a standard Cremona transformation based on the points of the set $\Gamma$, one obtains:
$${\rm Cr}_{\Gamma}(c_{I,t})=c_{J,t+1}.$$

Case 3. Finally, if  $\Gamma\subset S$ and $\Gamma$ contains neither the set $I$ nor its complement $I^c$, then:
$${\rm Cr}_{\Gamma}(c_{I,t})=c_{I,t}.$$

Analysing the three previous cases, we see that after applying a Cremona transformation to a given curve class $c_{I,t}$, the parameter $t$ might increase, stay the same or decrease by 1, while the dimension of the join that the irreducible curves with that class span remains invariant. Furthermore, by applying a suitable sequence of standard Cremona transformations, varying $\Gamma$ as appropriate, one proves that the set on the right-hand side of equation \eqref{proof: eff weyl orbit} is contained in the set on the left-hand one.

We prove the other inclusion, and hence conclude the proof, by showing that $t=0$ and $t=\lfloor n/2\rfloor$ are the minimum and maximum possible values of $t$, respectively, in the effective Weyl orbit \eqref{proof: eff weyl orbit}.

Indeed if $t=0$, and we choose $\Gamma$ such that $I\subset \Gamma$ as in Case 1 above,
the image of $c_{I,0}$ under the Cremona transformation $\Cr_\Gamma$ is non-effective. 
Thus, $t=0$ is the minimum value in the effective Weyl orbit \eqref{proof: eff weyl orbit}. On the opposite side, if $t=\lfloor n/2\rfloor$, we have that the index set $I$ must be such that 
\[
|I|= i+1-2\lfloor n/2\rfloor\le (n-1)+1-(n-1)=1.
\]
This implies that $I^c$ in $\{1,\ldots,n+3\}$ is of length $|I|\ge n+2$ and we cannot find $\Gamma$ as in Case 2 above. Thus, any Cremona transformation of $c_{I,t=\lfloor n/2 \rfloor}$ will never increase the value of $t$.

 \end{proof}

\begin{remark} \label{remark curve orbit} 
    In $X_s^n$ for $n+1\le s \le n+2$, the situation is much simpler. In this case, we must have $t=0$ and the joins of dimension $r$ are just the $r$-planes. Moreover, for $r\in\{1,\ldots, n-2\}$  we have that each curve class $c_{I,0}$ is alone 
    in its effective Weyl orbit, with $|I|=r+1$. In the case of $r=n-1$ with $|I|=n$ we have two cases:
    \begin{itemize}
        \item If $s=n+1$, then for each $j\in \{1, \ldots, n+1\}$ we have an effective Weyl orbit $\{e_j, c_{\widehat{I_j},0}\}$, where $\widehat{I_j}:=\{1,\ldots,n+1\}\setminus\{j\}$.
        \item If $s=n+2$, then we have a unique effective Weyl orbit $\{e_1, \ldots, e_{n+2}\}\cup\{c_{I,0} : |I|=n\}$.
    \end{itemize}
\end{remark}

\begin{theorem} \label{2 iff 3}
 Let $n+1\le s \le n+3$ and $X=X_s^n$. For any $1\le r\le n-2$, the Weyl $r$-planes of $X$ are all and only the joins of dimension $r$. The Weyl $(n-1)$-planes of $X$ are all and only the Weyl divisors.
\end{theorem}
\begin{proof} Let $1\le r \le n-1$. By definition, a Weyl $r$-plane is an element in the effective Weyl orbit of some $r$-plane $L_{\mathcal I}$, with $\mathcal I \subset \{1, \ldots, s\}$ a subset of length $|\mathcal I|=r+1$. We know, from Lemma \ref{who-sweeps-what}, that $L_{\mathcal I}$ is swept out by irreducible curves of class $c_{\mathcal I,0}$.  Let $\phi$ be any Weyl transformation, by Proposition \ref{proposition curve orbit} and Remark \ref{remark curve orbit} we know that $\phi(c_{\mathcal I,0})=c_{I,t}$ for some appropriate $I$ and $t$ with $|I|+2t-1=|\mathcal I|-1=r$, or $\phi(c_{\mathcal I,0})=e_i$ for some $i$, if $r=n-1$. Hence, $\phi(L_{\mathcal I})$ is swept out by curves of class $c_{I,t}$ or, possibly, by curves of class $e_i$ if $r=n-1$. By Lemma \ref{who-sweeps-what}, in the first case we obtain that $\phi(L_{\mathcal I})=J(L_I,\sigma_t)$, a join of dimension $|I|+2t-1=r$; in the second case we obtain that $\phi(L_{\mathcal I})=E_i$, an exceptional divisor. This concludes the proof.
\end{proof}

A further consequence of Proposition \ref{proposition curve orbit} is the fact that the cones of curves $\mathcal{C}_k$ are preserved by the effective Weyl action, {as described in the following theorem}.
For the description of $\mathcal C_0$  see Remark \ref{rmk35}.

\begin{theorem}[Cones of curves $\mathcal C_k$ in $X^n_{n+3}$]\label{strong-duality-thm-n+3}
	The cones of $k$-moving curves $\mathcal C_k$ in $X^n_{n+3}$ are generated by three effective Weyl orbits. More precisely:
	\begin{enumerate}[(a)]
		\item  the extremal rays of $\mathcal C_1$  are spanned by the elements of the Weyl orbits of $h-e_i$ and $e_i$;
		\item
		if $k\ge 2$, {the} generators of
		$\mathcal C_k$  are the elements of the effective Weyl orbits of
		$$h-e_i,\ e_i,\ (n-k)h- \sum_{|I|=n-k+1}e_i.$$ 
			\end{enumerate}
\end{theorem}

\begin{proof}

	The strong duality theorem for $X^n_{n+3}$, \cite[Theorem 5.5]{BDPS}, implies that the inequalities cutting out the cones of divisors $\mathcal{D}_k$ 
 give the rays of the cones of curves $\mathcal{C}_k$. By 	Proposition 5.2 of \cite{BDPS}, these are of four types:
	\begin{itemize}
		\item[(\RN 1)] $h$ and $h-e_i$, for every $i=1, \ldots, n+3$;
		\item[(\RN 2)] $nh-\sum_{i=1}^{n+3}e_i + e_j$, for every $j=1, \ldots, n+3$;
		\item[(\RN 3$_k$)] $ c_{I,t}=(|I| + (n+1)t-1)h - \sum_{i\in I} (t+1)e_i - \sum_{i \notin I} t e_i$,  for every $0\leq t \le n/2$ and $|I|=n-2t-k+1$;
		\item[(\RN 4)] $e_i$, for every $i=1, \ldots, n+3$.
	\end{itemize}

First notice that the curve classes in  (\RN 2) are in the Weyl orbit of $h-e_i$ for any $i\neq j$.

	Statement (a) follows since the  curve classes of type (\RN 3$_{k=1}$) correspond to those curve classes $c_{I,t}$ such that $|I|+2t-1=n-1$. By Proposition \ref{proposition curve orbit} they are all in the Weyl orbit of $e_i$ for any $i\in \{1,\ldots, n+3\}$.

	Similarly, statement  (b) is also a consequence of Proposition \ref{proposition curve orbit}. In this case, curve classes of type (\RN 3$_k$) form an effective Weyl orbit, having any curve class of the form $(n-k)h-\sum_{|I|=n-k+1} e_i$ in the orbit.
	\end{proof}

Notice that if $s\geq n+2$ then the curve classes $h-e_i$ and $h-e_j$ belong to the same Weyl orbit.
In fact if $\Delta:=\{i_i,\dots,i_{n}\}$ is a set of distinct indices not including $i$ and $j$, using \eqref{Cremona curves}, we see that $h-e_i$ is sent to $nh-\sum_{k\in\Delta\cup \{i,j\}}e_k$ by the Cremona transformation $\textrm{Cr}_{\Delta\cup\{j\}}$, 
which, in turn, is sent to $h-e_j$ by $\textrm{Cr}_{\Delta\cup\{i\}}$.
Therefore all $(0)-$Weyl lines lie in the same Weyl orbit. 
The same observation applies to the curve classes $e_i$ and $e_j$, if $s\geq n+2$.
Indeed, $e_i$ is sent to the  class
$(n-1)-\sum_{k\in\Delta\cup\{i\}}e_k$ by $\textrm{Cr}_{\Delta\cup\{i\}}$, 
which, in turn, is sent to $e_j$ by $\textrm{Cr}_{\Delta\cup\{j\}}$, where $\Delta$ in an index set as above.

If $s=n+2$ a Weyl group element will either stabilize the curve class $(n-k)h- \sum_{|I|=n-k+1}e_i$ or it will contract it. If $s=n+3$ any such curve class will belong to the same Weyl group orbit.  We obtain the following remark.

\begin{remark} 
We can make the statement of Theorem \ref{strong-duality-thm-n+3} more explicit by counting in how many orbits the generators of the cone $\mathcal C_k$ of $X^n_s$ fit, as shown in  Table \ref{table}.
\begin{table}[H]\label{table}
     \centering
    {\renewcommand{\arraystretch}{1.5}
     \begin{tabular}{|c|c|c|} 
        \cline{2-3}
        \multicolumn{1}{c|}{}
          &  $\boldsymbol{k\ge 2}$&  $\boldsymbol{k=1}$\\ \hline  
         $\boldsymbol{s=n+1}$&  $2n+2+\binom{n+1}{n-k+1}$& $2n+2$\\ \hline  
         $\boldsymbol{s=n+2}$&  $2+ \binom{n+2}{n-k+1}$& $2$\\ \hline  
         $\boldsymbol{s=n+3}$&  $3$& $2$\\ \hline 
     \end{tabular}}
     \vspace{5mm}
     \caption{Number of Weyl orbits whose elements generate $\mathcal C_k$ in $X_s^n$.}
     \label{tab:my_label}
 \end{table}
\end{remark}

 \section{Weyl cycles of dimension $r$} \label{sec WC}

In this section we characterize the Weyl cycles of dimension $r$ (see Definition \ref{Weyl-cycles})
 in $X_s^n$ for $n+1\le s \le n+3$.
 Theorem \ref{Weyl is join} shows that they are all and only the joins $J(L_I,\sigma_t)$ \eqref{join def} together with the exceptional divisors.

The Chow group $A_r(X_{s}^n)$ of $k$-dimensional  cycles on $X_{s}^n$ is generated by the class of a general $r$--plane $H_r$ and $r$-planes $E_{r,1}, \ldots, E_{r,s}$ contained in the exceptional divisors $E_1, \ldots, E_{s}$, respectively. Thus, any $r$--cycle class $V$ can be written as the following linear combination
$$\delta H_r - \sum_{i=1}^{s} \mu_i E_{r,i}.$$ 
In this context, we say that $V$ has $H_r$--degree $\delta$ and $E_{r,i}$--degree $\mu_i$, for every $1\le i \le s$.

\begin{remark} \label{rmk: joins}
A join 
$V=J(L_I,\sigma_t)$ is a subvariety of dimension $r=|I|+2t-1$.
It has:
\begin{itemize}
\item $H_{r}$--degree $\binom{t+n-r}{t}$,
\item $E_{r,i}$--degree $\binom{t+n-r}{t}$, for every $i\in I$,
\item $E_{r,i}$--degree $\binom{t+n-r-1}{t-1}$, for every $i\in \{1, \ldots, s\}\setminus I$.
\end{itemize}
Moreover it is swept out by the curve class
\[
c_V = (|I| + (n+1)t-1)h - \sum_{i\in I} (t+1)e_i - \sum_{i \notin I} t e_i,
\]
as proved in Lemma \ref{who-sweeps-what}. We also recall that for $n+1\le s \le n+2$, joins $J(L_I,\sigma_t)$ are just $r$-planes with $|I|=r+1$ and $t=0$ in this notation. Here we are using the convention $\binom{a}{0}=1$ for any $a>0$, and $\binom{a}{b}=0$ if $b<0$. 
\end{remark} 

\begin{lemma} \label{curve dec}
Let  $V$ and $W$ be two joins in $X^{n}_{s}$ for $n+1\le s \le n+3$, and let $c_V$, $c_W$ be the corresponding curve classes. If $W\subset V$, then $c_V=c_W+c_{W,V}$ where $c_{W,V}$ is a sum of $(0)$-Weyl lines, in particular $c_{W,V}$ is $0$-moving.

\end{lemma}
\begin{proof}
We first prove the statement when $\dim W= \dim V-1$. In this case, if $V=J(L_I,\sigma_t)$ then
\begin{enumerate}
\item either $W=J(L_{I\setminus \{i\}},\sigma_t)$, for some $i\in I$, in which case
		$$ c_V = c_W + (h-e_i),$$
so, $c_{W,V}=h-e_i$ is the $(0)$-Weyl line;
\item or, if $s=n+3$,  $W=J(L_{I\cup \{i\}}, \sigma_{t-1} )$, for some $i \in \{1, \ldots, n+3\}\setminus I$, in which case
		$$ c_V= c_W + \left( nh - \sum_{j=1, j\neq i}^{n+3} e_j \right)$$
so, $c_{W,V} = nh - \sum_{j=1, j\neq i}^{n+3} e_j $ is the $(0)$-Weyl line.
\end{enumerate}

Assume now that $\dim W= \dim V - a$ for some $a \geq 2$ and consider the flag 
\[
V= W_0 \supset W_1 \supset \ldots \supset W_a = W,
\]
where each $W_i$ is a join and $\dim W_i -1 = \dim W_{i+1}$. By the previous discussion, for every $i\in \{0, \ldots, a-1\}$ there exists a $0$-moving curve $c_{W_{i+1},W_i}$ such that $c_{W_i} = c_{W_{i+1}} + c_{W_{i+1},W_i}$. Hence, applying this recursively we conclude that
\[
c_V = c_W + \left( \sum_{i=1}^a c_{W_i, W_{i-1}} \right),
\]
with $c_{W,V}= \sum_{i=1}^a c_{W_i, W_{i-1}}$ being a sum of $(0)$-Weyl lines and, in particular, it is $0$-moving since it is a sum of $0$-moving curves.
\end{proof}

\begin{definition} \label{orthogonal join}
Given an effective divisor $D$ on $X_{s}^n$, we say that a join $V=J(L_I,\sigma_t)$ is orthogonal to $D$ if $D\cdot c_V = 0$ or, equivalently, if $\kappa_V(D)=0$, with $\kappa_V(D)$ as in \eqref{k_J}.
\end{definition}

\begin{lemma} \label{k_J(D)}
Let $D=E_{\mathcal I,\sigma_\tau}:=(\tau+1)H-\sum_{i\in \mathcal I} (\tau+1)E_i - \sum_{i\notin \mathcal I} \tau E_i$ be a Weyl divisor in $X_{s}^n$ and $V:=J(L_I,\sigma_t)$ a join, then
\[
\kappa_V(D) = \tau-t+1-|I\setminus \mathcal I|.
\]
In particular, V is orthogonal to D if and only if
\[ 
|I\setminus \mathcal I| = \tau+1-t.
\]
\end{lemma}
\begin{proof}
Notice that, since $D$ is a divisor, it must be $|\mathcal I| + 2\tau -1 = n-1$. Thus, one can compute
\begin{align*}
\kappa_V(D) 	=& \; - D\cdot c_V  \\
               =& \; -(\tau+1)(|I|+(n+1)t-1)+(\tau+1)(t+1)|\mathcal I\cap I|+ (\tau+1)t|\mathcal I\setminus I| \\
              & + \tau(t+1)|I\setminus \mathcal I| + \tau t (n+3-|\mathcal I \cup I| )\\
        	= & \; \tau-t+1-|I\setminus \mathcal I|.
\end{align*}
The second statement follows from Definition \ref{orthogonal join}
\end{proof}

\begin{proposition} \label{thm: int ort}
Let $D=E_{\mathcal I,\sigma_\tau}$ be a Weyl divisor on $X_s^n$ with $n+1\le s \le n+3$ and $V=J(L_I,\sigma_t)$ a join orthogonal to $D$, if $t=0$ then
\[
D\cap V= \left(\bigcup_{i\in I\setminus \mathcal I} J(L_{I\setminus \{ i \}}, \sigma_t)\right),
\]
while, if $t\ge1$, we have 
\[
D\cap V= \left(\bigcup_{i\in I\setminus \mathcal I} J(L_{I\setminus \{ i \}}, \sigma_t)\right) \cup \left(\bigcup_{i\in \mathcal I\setminus I} J(L_{I\cup \{i \}}, \sigma_{t-1})\right).
\]
\end{proposition}
\begin{proof}
Since $V$ is orthogonal to $D$ by hypothesis, $V$ is not a subvariety of $D$. Thus, if $V$ has dimension $r$, the intersection of $D$ and $V$ is a union of dimension $r-1$ subvarieties. The theorem specifies which are these varieties: the ones on the right hand side of the equation. Indeed they are of dimension $r-1$ and they are clearly contained in $V$. Moreover, using Lemma \ref{k_J(D)} and by the orthogonality of $V$ and $D$, we have that
\begin{itemize}
\item For $i\in I\setminus \mathcal I$,
\begin{align*}
\kappa_{ J(L_{I\setminus \{ i \}}, \sigma_t)}(D) &=\tau-t+1-|(I\setminus \{ i \}) \setminus \mathcal I| \\
								& = \tau-t+1-(|I\setminus \mathcal I|-1)=\kappa_V(D)+1 = 1.
\end{align*}
\item For $i\in \mathcal I\setminus I$, and $t\ge 1$
\begin{align*}
\kappa_{ J(L_{I\cup \{ i \}}, \sigma_{t-1})}(D) &=\tau-(t-1)+1-|(I\cup \{ i \}) \setminus \mathcal I| \\
								& = \tau-t+1+1-|I\setminus \mathcal I|=\kappa_V(D)+1 = 1.
\end{align*}
\end{itemize}
This proves that each join on the right hand side is contained in $D$ once. Hence, we have proven the containment "$\supseteq$". 

Notice that $D\cap V$ is a $r-1$--cycle with 
\begin{enumerate}[(a)]
\item $(\tau+1)\binom{t+n-r}{t}$ as $H_{r-1}$--degree,
\item $(\tau+1)\binom{t+n-r}{t}$ as $E_{r-1, j}$--degree for every $j \in \mathcal I\cap I$,
\item \label{1 t=0} $(\tau+1)\binom{t+n-r-1}{t-1}$ as $E_{r-1, j}$--degree for every $j \in \mathcal  I\setminus I$,
\item $\tau\binom{t+n-r}{t}$ as $E_{r-1, j}$--degree for every $j \in I\setminus \mathcal I$,
\item \label{2 t=0} $\tau\binom{t+n-r-1}{t-1}$ as $E_{r-1, j}$--degree for every $j \notin \mathcal I\cup I$.
\end{enumerate}
Observe that, for $t=0$ the $E_r,j$--degree is 0 in cases \ref{1 t=0} and \ref{2 t=0}.

To conclude the proof we show that the right hand side is an $r-1$--cycle with the same class. We do so for a general $t$, the case $t=0$ follows the same computations but only considering the first union in the right hand side. Thus, it is left to the reader.

We start by checking the $H_{r-1}$--degree. Following Remark \ref{rmk: joins}, for every $i\in I\setminus \mathcal I$ each join $ J(L_{I\setminus \{ i \}}, \sigma_t)$ has $H_{r-1}$--degree $\binom{t+n-r+1}{t}$, while for every $i \in \mathcal I\setminus I$ the joins $ J(L_{I\cup \{ i \}}, \sigma_{t-1})$ have $H_{r-1}$--degree $\binom{t+n-r}{t-1}$. Thus, their union has $H_{r-1}$--degree

{\small
\begin{equation} \label{eq: degree joins}
\sum_{i\in I\setminus \mathcal I} \binom{t+n-r+1}{t} + \sum_{i \in \mathcal I\setminus I} \binom{t+n-r}{t-1} = |I\setminus \mathcal I|\binom{t+n-r+1}{t} + |\mathcal I\setminus I|\binom{t+n-r}{t-1}.
\end{equation}
}

Since $D=J(L_{\mathcal I}, \sigma_\tau)$ and $V=J(L_I,\sigma_t)$ are varieties of dimension $n-1$ and $r$, respectively. We have that $|\mathcal I|=n-2\tau$ and $|I|=r-2t+1$. Moreover, since $V$ is orthogonal to $D$, by Lemma \ref{k_J(D)} we have that $|I\setminus \mathcal I|=\tau+1-t$.  Therefore,
\[
|\mathcal I\cap I| = |I| - |I\setminus \mathcal I| = r-2t+1-(\tau+1-t) = r-t-\tau,
\]
and
\[
|\mathcal I\setminus I|= |\mathcal I|-|\mathcal I\cap I|= n-2\tau - (r-t-\tau) = n+t-r-\tau.
\]
Thus,
\[
|I\setminus \mathcal I| + |\mathcal I\setminus I| = \tau +1-t +n+t-r-\tau=n-r+1.
\]
This, together with the usual binomial identity $\binom{t+n-r+1}{t}=\binom{t+n-r}{t} + \binom{t+n-r}{t-1}$, we have that equation \ref{eq: degree joins} is equal to
\[
|I\setminus \mathcal I|\binom{t+n-r}{t} + \left(|I\setminus \mathcal I| + |\mathcal I\setminus I|\right)\binom{t+n-r}{t-1} 
\]
\[
= (\tau+1-t)\binom{t+n-r}{t} + (n-r+1)\binom{t+n-r}{t-1} 
\]
\[
=  (\tau+1-t) \binom{t+n-r}{t} + t \binom{t+n-r}{t} 
\]
\[
=(\tau+1-t + t)\binom{t+n-r}{t} = (\tau+1)\binom{t+n-r}{t}.
\]

As for the $E_{r-1,j}$--degrees, we have the following cases:
\begin{enumerate}[(a)]
\item If $j \in \mathcal I\cap I$, we have that each join $ J(L_{I\setminus \{ i \}}, \sigma_t)$ has $E_{r-1,j}$--degree $\binom{t+n-r+1}{t}$  for every $i\in I\setminus \mathcal I$, and the joins $ J(L_{I\cup \{ i \}}, \sigma_{t-1})$ have $E_{r-1,j}$--degree $\binom{t+n-r}{t-1}$  for every $i \in \mathcal I\setminus I$.  Following the exact same calculation as before, we conclude that the right hand side has $E_{r-1,j}$--degree $(\tau+1)\binom{t+n-r}{t}$.

\item If $j \in \mathcal I\setminus I$ then, the right hand side has $E_{r-1,j}$--degree
\[
\binom{t+n-r}{t-1} + \sum_{i\in I\setminus \mathcal I} \binom{t+n-r}{t-1} + \sum_{i \in \mathcal I\setminus I, i\neq j} \binom{t+n-r-1}{t-2},
\]
and a similar calculation shows that this adds up to $(\tau+1)\binom{t+n-r-1}{t-1}$.
\item If $j \in I\setminus \mathcal I$ then, the right hand side has $E_{r-1,j}$--degree
\[
\binom{t+n-r}{t-1} + \sum_{i\in I\setminus \mathcal I, i \neq j} \binom{t+n-r+1}{t} + \sum_{i \in \mathcal I\setminus I} \binom{t+n-r}{t-1},
\]
which adds up to $\tau\binom{t+n-r}{t}$.
\item If $j \notin \mathcal I\cup I$ then, the right hand side has $E_{r-1,j}$--degree
\[
 \sum_{i\in I\setminus \mathcal I} \binom{t+n-r}{t-1} + \sum_{i \in \mathcal I\setminus I} \binom{t+n-r-1}{t-2},
\]
which adds up to $\tau\binom{t+n-r-1}{t-1}$.
\end{enumerate}
Since the right hand side is contained in the left hand side and both have the same class, we conclude that they are equal.
\end{proof}

\begin{proposition} \label{joins are weyl cycles}
    Let $V=J(L_I,\sigma_t)$ be a join of dimension $r$ in $X_s^n$, for $n+1\le s \le n+3$. Then $V$ is a Weyl $r$-cycle.
\end{proposition}
\begin{proof}
    Without loss of generality we may assume that $I=\{1,\ldots, a\}$ for some $a\in\{0,\ldots,n\}$ such that $\dim V=|I|+2t-1=a+2t-1=r$; where $a=0$ means that $I$ is the empty set.

    To prove the statement we will show that $V$ is an irreducible component of 
    \[
    D_{a+1}\cap \ldots \cap D_{n+1-2t},
    \]
    where $D_j=E_{\mathcal I_j,\sigma_t}$ (as defined in Lemma \ref{k_J(D)}) with $\mathcal I_j=\{1, \ldots, n+1-2t\}\setminus \{j\}$. Indeed, notice that they are pairwise orthogonal divisors since for every $j_1\neq j_2$
\[
\langle D_{j_1}, D_{j_2} \rangle = \kappa_{D_{j_1}}(D_{j_2}) = t-t+1-|\mathcal I_{j_1}\setminus \mathcal I_{j_2}| = 1 - |\{j_2\}| =0 .  
\]
    
    If $r=n-1$, then $V=D_{n+1-2t}$ and there is nothing to prove. 
    
    By induction, assume the statement to be true for joins of dimension $r+1$, and let $V=J(L_I,\sigma_t)$ be of dimension $r\le n-2$. 
    
    Consider $V'=J(L_{I\cup\{a+1\}}, \sigma_t)$, which is a join of dimension $r+1$. By the induction hypothesis we know that $V'$ is an irreducible component of
    \[
    D_{a+2}\cap \ldots \cap D_{n+1-2t}.
    \]
    Following Lemma \ref{k_J(D)}, observe that
    \[
    \kappa_{V'}(D_{a+1})=t-t+1-\left|\{1,\ldots,a+1\}\setminus \mathcal I_{a+1} \right|=1-|\{a+1\}|=0,
    \]
    which means that $V'$ is orthogonal to $D_{a+1}$. By Proposition \ref{thm: int ort} it follows that $V$ is an irreducible component of $D_{a+1}\cap V'$. Hence,  $V$ is an irreducible component of
    \[
    D_{a+1}\cap D_{a+2} \cap \ldots \cap D_{n+1-2t}.
    \]

\end{proof}

In particular the next theorem answers affirmatively a question posed in \cite[Conjecture 2]{BDP-Ciro}.
\begin{theorem}\label{Weyl is join}
Weyl $r$-cycles in $X_s^n$ are all and only the joins of dimension $r$, for $n+1\le s\le n+3$. 
\end{theorem}
\begin{proof}
By  Proposition \ref{joins are weyl cycles} we know that joins of dimension $r$ are Weyl $r$-cycles. We prove that if $D_1, \ldots, D_l$ is a list of pairwise orthogonal Weyl divisors, their intersection is a union of joins.

We use induction on $l$. 
If $l=1$ there is nothing to prove. The case $l=2$ is a particular case of Proposition \ref{thm: int ort} when $V$ is itself a divisor. Assume the statement to be true for $l-1$, then
\[
D_1 \cap \ldots \cap D_l = \left(D_1 \cap \ldots \cap D_{l-1}\right) \cap D_l = \left(\bigcup_{i=1}^m V_i\right) \cap D_l = \bigcup_{i=1}^m(V_i \cap D_l),
\]
where $V_1, \ldots, V_m$ are the joins whose union is the intersection $D_1 \cap \ldots \cap D_{l-1}$. We conclude by showing that $V_i \cap D_l$ is a union of joins for every $i\in \{1, \ldots, m\}$.

On one hand, since $V_i$ is a join contained in $D_1$, by Lemma \ref{curve dec} there exists a $0$-moving curve $c_{V_i, D_1}$ such that $c_{D_1}=c_{V_i} + c_{V_i, D_1}$. On the other hand, by the orthogonality of $D_1$ and $D_l$, we have that
\[
0= c_{D_1}\cdot D_l = c_{V_i}\cdot D_l + c_{V_i, D_1}\cdot D_l.
\]
We know that $c_{V_i, D_1}\cdot D_l\ge 0$, since $c_{V_i, D_1}$ is $0$-moving and $D_l$ is an effective divisor. Hence, it must be that $c_{V_i}\cdot D_l\le 0$. If  $c_{V_i}\cdot D_l< 0$, then $V_i$ is contained in $D_l$ and, therefore $V_i \cap D_l = V_i$, which is a join. Otherwise,  $c_{V_i}\cdot D_l = 0$, which means that $V_i$ is a join orthogonal to $D_l$ and by Proposition \ref{thm: int ort} we know that their intersection is a union of joins.
\end{proof}

\section{On Gale duality and Weyl actions}
For the sake of completeness, we review in this section the birational geometry of the spaces $X^4_8$ and $X^3_7$. In particular we describe the cones of $k$-moving curves in Theorems \ref{strong-duality-thm-X^4_8} and \ref{strong-duality-thm-X^3_7}.
\label{gale section}

\subsection{Blow up of $\PP^4$ in $8$ points}
The birational geometry of $X^4_8$ was described by Mukai \cite{Mukai05}  and by Casagrande, Codogni and Fanelli \cite{CCF}, by means of the Gale duality between $X:=X^4_8$ and the Del Pezzo surface $S:=X^2_8$, a projective plane blown up in eight general points. 
 Further properties are given in  \cite{Xie}, where in particular the role of the anticanonical curve is analyzed.

Gale duality, see \cite{Dolgachev-Ortland, Eisenbud-Popescu}, arises from a bijection between sets of $8$ general
points in $\PP^2$, $\{p_1,\ldots, p_8\}$, and in $\PP^4$, $\{q_1,\ldots q_8\}$  up to projective equivalence. 
We will dedicate this section to unify the approaches of \cite{CCF,Mukai05} and the Weyl group actions on curve classes of \cite{DM2}. In particular we show in Proposition \ref{prop} that the Weyl actions on $S$ and $X$ are compatible with the Gale duality.

Let $\{h, e_1,\ldots, e_8\}$ denote a basis for $\N_1(X)_\mathbb R$,
and let 
$\{\alpha, \beta_1, \dots \beta_8\}\in H^2(S, \mathbb{R})$ be a basis for $\Pic(S)$, where $\alpha$ is a general line class and $\beta_i$ is the class of the exceptional divisor 
$E_i$ and set
$\beta:=\sum_{i=1}^8 \beta_i$. 

Mukai proves that $X$  is isomorphic to the moduli space $Y$ of rank two torsion free sheaves with first Chern class $-K_S$ and second Chern class $2$, 
semistable with respect to $-K_S + 2{\alpha}$.
Between the spaces $X$ and $S$  
there exists the following isomorphism of real vector spaces,
\cite[Proposition 5.8]{CCF},
$$\rho:  \N_1(X)_\mathbb R \to H^2(S, \mathbb{R})$$
such that $\rho(h)=\beta-\alpha$ and  $\rho(e_i)=-2\beta_i+\beta-\alpha.$ In particular, its inverse
\begin{equation}\label{eq}
	\rho^{-1}:H^2(S, \mathbb{R}) \to \N_1(X)_\mathbb R
\end{equation} 
is determined by $\rho^{-1}(\alpha)=3h-\frac12\sum_{i=1}^8 e_i$ and $\rho^{-1}(\beta_i)=\frac12(h-e_i)$. Furthermore,
$$\rho^{-1}(K_S)=F,$$ 
where $F:=5h-\sum_{i=1}^{8}e_i$ is the anticanonical curve class, see
Definition \ref{DEF anticanonical curve}.

Let $W_S:=W^2_8$ denote the Weyl group on $S$ and $W_X:=W^4_8$ 
denote the Weyl group of $X$, two finite groups. 
As a straightforward application of the Weyl group action on curve classes, see \eqref{Cremona curves}, we make explicit the following result, which was already known.

\begin{proposition}\label{prop}\cite{Mukai05}\cite[Lemma 5.10]{CCF}
Denote by $\Gamma^S_{i,j,k}:=\{p_i,p_j,p_k\}$ and $\Gamma^X_{i,j,k}:=\{q_l: l\neq i,j,k\}$
and
let $w\in \N_1(S)_\mathbb R$ be a divisor on $S$. 
The map from
$W_X$ to $W_S$
sending  ${\rm Cr}_{\Gamma^S_{i,j,k}}$ to ${\rm Cr}_{\Gamma^X_{i,j,k}}$ is an isomorphism.
Moreover we have
	$${\rm Cr}_{\Gamma^S_{i,j,k}}(\rho^{-1}(w))=\rho^{-1}({\rm Cr}_{\Gamma^X_{i,j,k}}(w)).$$
\end{proposition}
\begin{remark}\label{rmk-lines}
Recall that by \cite{DM2}, for any Del Pezzo surface of type $X^2_s$ and any $i\in\{-1,0,1\}$ the notions of $(i)-$curve and $(i)-$Weyl line coincide,  see Definitions
 \ref{(i) curves} and
\ref{(i) weyl lines}.
{Explicitly we have that 
the $(-1)$-curves ($(0)$-curves, $(1)$-curves respectively) on $S$
are the elements of the effective Weyl  orbit of $\beta_i$ ($\alpha-\beta_i$, $\alpha$, respectively). }
\end{remark}

\begin{proposition} 
	Define the map
	$$\eta:H^2(S, \mathbb{R}) \to \N_1(X)_\mathbb R$$	
	by $\eta(w):=\rho^{-1}(2w{-}K_S)$ for $w\in H^2(S, \mathbb{R})$. Then
	\begin{enumerate}
		\item $2\rho^{-1}$ sends $(-1)-$curves on $S$ to $(0)-$Weyl lines on $X$.
		\item $\eta$ sends $(1)-$curves on $S$ to $(1)-$Weyl lines on $X$.
		\item $\eta$ sends $(0)-$curves on $S$ to the Weyl orbit of $e_i$ on $X$.
		\item $\eta$ sends $(-1)-$curves on $S$ to the Weyl orbit of $2h-\sum_{|I|=3} e_i$ on $X$.
	\end{enumerate}
\end{proposition}

\begin{proof} 	  The proof is an application of Proposition \ref{prop}, {Remark \ref{rmk-lines} and the  following computations:}
	\begin{enumerate}
		\item $2\rho^{-1}(\beta_i)=h-e_i$ 
		\item $\eta(\alpha)=h $  
		\item $\eta(\alpha-\beta_i)=e_i$ 
\item $\eta(2\alpha-\sum_{i=1}^5\beta_i)=2h-\sum_{i\in \{6,7,8\}}e_i$ .
	\end{enumerate}

{From Proposition \ref{prop} it follows 
that 	${\rm Cr}_{\Gamma^S_{i,j,k}}(\eta^{-1}(w))=\eta^{-1}({\rm Cr}_{\Gamma^X_{i,j,k}}(w)).$
Hence, to conclude  the proof of statement $(4)$, it is enough to notice that $\Cr_{\Gamma_{6,7,8}}^S(\alpha-\beta_1-\beta_2)=2\alpha- \sum_{i=1}^5\beta_i$. }
\end{proof}

 In the following remark we translate the description of 
 the birational geometry of $X$ of \cite{CCF}
 into the language developed in our work.

\begin{remark}\label{rem}
	\begin{enumerate}
	\item The effective cone of $X$ is described by hyperplanes in
	$$\{\rho^{-1}(w)^{\perp} \text{ where } w \text{ is a $(-1)$-curve on S }\}\cup \{\eta(w)^{\perp} \text{ where } w \text{ is a $(1)$-curve on S}\}.$$ 
	
	\item The movable cone, $Mov(X)$  is described by hyperplanes in
	$$\{\rho^{-1}(w)^{\perp} \text{ where } w \text{ is a $(-1)$-curve on S }\}\cup \{\eta(w)^{\perp} \text{ where } w \text{ is a $(0)$-curve on S}\}.$$
	
	\item The Mori chamber decomposition of $X$ is given by hyperplanes perpendicular to the Image of $\eta$, that is 
	$$\{\eta(w)^{\perp},\text{ where } w \text{ is a $(-1)$-curve on S}\}\cup \{\eta(w)^{\perp},\text{ where } w \text{ is a $(0)$-curve on S}\}.$$
	\end{enumerate}
We remark that in case (4) of Proposition \ref{prop} some  elements of the Weyl orbit are not effective, in particular the negative of $(-1)$-Weyl lines and the negative of rational normal curves through seven points on $X$ arises as images via $\eta$ of $(-1)$-curves on $S$. 
Indeed $\eta(\alpha-\beta_i-\beta_j)=e_i+e_j-h$ and
$\eta(\beta_j)=\sum_{i=1}^8 e_i-e_j-4h$. 
\end{remark}

{In the following result, which is analogous to  \Cref{strong-duality-thm-n+3}, we 
describe the cones of $k$-moving curves.}

\begin{theorem}[Cones of curves in $X^4_8$]\label{strong-duality-thm-X^4_8}
In $X=X^4_8$, the extremal rays for the cones of curves  $\mathcal C_k$ are effective Weyl curve orbits. In particular,
	\begin{enumerate}[(a)]
		\item the extremal rays of $\mathcal C_0$  are the elements of the Weyl orbits of $h$ and $h-e_i$,
		\item the extremal rays of $\mathcal C_1$  are the elements of the Weyl orbits of $h-e_i$ and $e_i$,
		\item
		the extremal rays of $\mathcal C_2$ are the effective elements of the Weyl orbits of
		$$h-e_i,\ e_i,\ 2h- \sum_{|I|=3}e_i.$$ 
  \item
		the extremal rays of $\mathcal C_3$ are the effective elements of the Weyl orbits of
		$$h-e_i-e_j,\ e_i.$$ 
	\end{enumerate}
\end{theorem}
\begin{proof}
{By the strong duality theorem for $X^4_{8}$, \cite[Theorem 5.7]{BDPS}, we know that  $\mathcal{C}_k=\mathcal{D}_k^\vee$.}
	The cones $\mathcal{D}_0=\Eff(X)$
	and $\mathcal D_1=\Mov(X)$ have been computed by \cite{CCF}. In particular $\mathcal D^\vee_0$ and $\mathcal D^\vee_1$ are generated by the curves in (a) and (b), respectively.
	It is easy to see that the cone $\mathcal D_2^\vee$ is generated by the curves in (c).
 Finally, the cone $\mathcal{D}_3$ is known by \cite{DP-positivityI}.
\end{proof}

\subsection{Blow up of $\PP^3$ in $7$ points}
If $X=X^3_7$ analogous results hold and are easier to be proved. 
The description of the Weyl cycles in this case can be found in \cite[Section 4]{BDP-Ciro}.
In particular we have the following description of the cones $\mathcal{C}_k$.

\begin{theorem} [Cones of curves in $X^3_7$] \label{strong-duality-thm-X^3_7}
In $X=X^3_7$, the extremal rays for the cones of curves  $\mathcal C_k$ are effective Weyl curve orbits. In particular,
	\begin{enumerate}[(a)]
		\item The extremal rays of $\mathcal C_0$  are the elements of the Weyl orbits of $h$ and $h-e_i$,
		\item The extremal rays of $\mathcal C_1$  are the elements of the Weyl orbits of $h-e_i$ and $e_i$.
	\end{enumerate}
\end{theorem}

\section{Weyl $r$-planes in non-Mori dream space cases}\label{section:infinite Weyl}

\subsection{Infinity of Weyl orbits}\label{section infinity}

Assume now that $X^n_s$ is not a Mori dream space.
It is well known that there are infinitely many Weyl divisors (i.e.\ Weyl $(n-1)$-planes) on $X$.
In \cite{DM2} the authors proved there are infinitely many $(-1)$-Weyl lines (i.e.\ Weyl $1$-planes, {see Definition \ref{(i) weyl lines}}) if $s\geq n+5$ (and $n\ge3$), and finitely many $(-1)$-Weyl lines if $s\leq n+4$. They also proved that the space $X^5_9$ has finitely many $(-1)$-curves, see Definition \ref{(i) curves}, even if it is not a Mori dream space.

In this section we will prove \Cref{infinity}, which claims that the orbits  
of Weyl $r$-planes, 
for any $2\leq r\leq n-1$, contain infinite elements.

The next proposition was originally used  to give an alternative proof, using the Weyl action on curve classes, that $X^n_s$ is not a Mori Dream Space if $s\geq n+4$ and $n\geq 5$, or $s\geq n+5$ and $n\in \{3,4\}$ or $s\geq 9$ for $n=2$.
We will further use it in this section to prove the infinity of Weyl $r$-planes.

\begin{proposition}[\cite{DM2}] \label{recursionCremona} Let $s\in \{n+4, n+5\}$ and  $L = (d;m_1 \leq m_2 \leq \ldots \leq m_{s})$ denote the curve class $c=dh-\sum_{i=1}^{s}m_ie_i$ in $A_1(X^n_{s})$
	with positive degree and non-negative multiplicities in nondecreasing order.
	Let $I \subset \{1,\ldots,s\}$ have length $n+1$.
\begin{enumerate} 
\item Let $s=n+4$ and $n\geq 5$.
	Assume that the three indices missing from the set $I$
	 are $1, k, \ell$ with $4 \leq k$ and $k+3 \leq \ell \leq n+2$.
	This implies in particular that $1 \notin I$; $2, 3 \in I$;
	$m_{n+3}, m_{n+4} \in I$.
	Assume that
	\begin{equation}\label{recursioninequality}
		d > \sum_{i \in I} m_i, 
		\text{ or } d = \sum_{i \in I} m_i \text{ and } m_{n+4} > m_1.
	\end{equation}
	Then, for $\Gamma=\{1, \ldots, n+1\}$, the Cremona image $\Cr_\Gamma(L)$
	has the form  $\Cr_\Gamma(L) = (d’;m_1’\leq\ldots\leq m_{n+4}')$
and the parameters $d',m_i'$ also satisfy (\ref{recursioninequality}).
	Moreover in this case $d' > d$.

 \item Let $s=n+5$ and $n=3$ or $n=4$. Furthermore, assume 

 \begin{equation}\label{recursioninequality2}
d>m_3 + m_4 +\sum_{i=7}^{n+5} m_i \text{  or } d=m_3+m_4+\sum_{i=7}^{n+5} m_i \text{ and } m_{n+5}>m_1. 
\end{equation}
Then the Cremona images $\Cr_\Gamma(L)$ also has parameters satisfying \eqref{recursioninequality2} and the degree increases.
 
 \end{enumerate}
\end{proposition}

\begin{proof} Part (1) is \cite{DM2}[Lemma 5.4]. Part (2) follows from \cite{DM2}[Lemma 5.4], see page 32, relation (5.6).
\end{proof}
Recall from \eqref{Cremona zero}  the following notation, for any $\mathcal I$ of lenght $r+1$
$$ c_{\mathcal I}=rh-\sum_{i\in \mathcal I} e_i.$$
The curves of classes $c_{\mathcal I}$ for $|\mathcal I|=2$ are the $(-1)$-Weyl lines, see Definition \ref{(i) weyl lines}.

In the following lemma we study the orbit of a curve of class $c_{\mathcal I,0}$ for any $|\mathcal I|\ge3$, i.e.\ for $r\ge2$.

\begin{proposition}\label{cor} Let $X^n_{s}$ be not a Mori dream space and $2\leq r \leq n-1$.
 Then the  effective Weyl orbit of the curve class $ c_{\mathcal I}$, for any $\mathcal I$ of length $r+1$,
is infinite. 
	\end{proposition}

\begin{proof} 
It is enough to prove the statement for $s=n+4$ with $n\geq 5$ and for $s=n+5$ with $n\in \{3,4\}$.

Up to permutation of the points, it is enough to prove the claim for $\mathcal I=\{n+4-r,\ldots,n+4\}$. Let us call 
$c_r:=c_{\mathcal I}=rh-\sum_{i=n+4-r}^{n+4} e_i$
\begin{enumerate}
\item Assume first $s=n+4$ and $n\geq 5$.
 
 Apply Proposition \ref{recursionCremona} part (1) to the curve class $c_r$, with $m_j=0$ for $1\leq j\leq n+3-r$ while $m_i=1$ for $n+4\geq i \geq n+4-r$. For arbitrary $r$ with $2\leq r\leq n-1$ we can chose $k=4$ and $l=n+2$ so $m_4=0$ and $m_{n+2}=1$.
Moreover since $r\geq 2$ then $l=n+2\geq n+4-r$, that is, $l\in \mathcal{I}$ and also $l\notin I$, where $I$ is the index set defined in Proposition \ref{recursionCremona}. 
 
 If $2\leq r\leq n-1$, the curve class $c_r$ satisfies the requirement	$d = \sum_{i \in I} m_i$ and $m_{n+4}>m_1$.

\item We assume $s=n+5$ and $n\in\{3,4\}$. We further apply Proposition \ref{recursionCremona} part (2) to the curve class $c_2=2h-e_6-e_7-e_8$ in $X^3_8$ with $m_j=0$ for $j\leq 5$ and $m_i=1$ for $i\geq 6$. We see that $c_2$ satisfies the  requirement \ref{recursioninequality2}
$$d=m_3+m_4+\sum_{i=7}^{n+5} m_i \text{ and } m_{n+5}>m_1.$$ 
Similarly, $c_3=3h-e_6-e_7-e_8-e_9$ in  $X^4_9$ satisfies 
$$d=m_3+m_4+\sum_{i=7}^{n+5} m_i \text{ and } m_{n+5}>m_1.$$

However, the curve $c_2=2h-e_7-e_8-e_9$ in  $X^4_9$ does not satisfy the hypothesis of Proposition \ref{recursionCremona}, namely it fails inequality \ref{recursioninequality2}. Furthermore, we leave it to the reader to check that the curve $d=8$, $m_1=0$, $m_2= m_3=m_4=1$ and $m_j=2$ for $5\leq j \leq 9$ is in the Weyl orbit of $c_2$ in  $X^4_9$ and satisfies the inequality  \ref{recursioninequality2}. We conclude that the curve class $c_2$ also has an infinite Weyl orbit in  $X^4_9$.
 \end{enumerate}
\end{proof}

\begin{corollary}\label{exc weyl}
Let $X^n_{s}$ be not a Mori dream space.
Then there are infinitely many exceptional Weyl lines. 
\end{corollary}

\begin{proof}
Exceptional Weyl lines are elements of the Weyl orbit of $e_i$. 
If $\mathcal I \subset\{1, \ldots,s\}$ has length $n+1$ and $i\in \mathcal I$, then
$\Cr_{\mathcal I}(e_i)=c_{\mathcal I\setminus\{i\}}$.
We apply Proposition \ref{cor} to the curve class $c_{\mathcal I\setminus\{i\}}$.
    \end{proof}

\begin{lemma}\label{rmk weyl orbit curves}
Assume that $s\ge n+3$, $1\le r\le n-1$.  
Given $\mathcal{I},\mathcal{J}\subset\{1,\dots,s\}$ with $|\mathcal{I}|=|\mathcal{J}|$, then there is $w\in W^n_s$ such that $w(c_{\mathcal{I
}})=c_{\mathcal{J
}}$.
\end{lemma}
\begin{proof}
This follows from Proposition \ref{proposition curve orbit}. Indeed let $\mathcal I, \mathcal J \subset\{1,\dots,s\}$ such that $|\mathcal{I}|=|\mathcal{J}|$, and let $\mathcal I \setminus \mathcal J = \{i_1, \ldots, i_l\}$ and $\mathcal J \setminus \mathcal I = \{j_1, \ldots, j_l\}$. By Proposition \ref{proposition curve orbit} we know that $c_\mathcal I$ is in the same effective Weyl orbit as $c_{\mathcal I_1}$, for $\mathcal I_1 = (\mathcal I\setminus \{i_1\}) \cup \{j_1\}$. In turn, $c_{\mathcal I_1}$ is in the same Weyl orbit as $c_{\mathcal I_2}$, for $\mathcal I_2 = (\mathcal I_1 \setminus \{i_2\}) \cup \{j_2\}$. We can repeat this process until we reach $\mathcal I_l=\mathcal J$.
\end{proof}

\begin{proposition}\label{1-1}
Assume that $s\ge n+3$ and fix $1\le r\le n-1$.
Then there is a  $1$-$1$ correspondence between the set of all Weyl $r$-planes of $X^n_s$ and  the effective orbit of $c_\mathcal{I}$, for any set $\mathcal{I}\subset \{1,\dots,s\}$, with $|\mathcal{I}|=r+1$.
 \end{proposition}
 \begin{proof}
Fix $\mathcal{I}\subseteq\{1,\dots,s\}$, a set of length $r+1$. 
If $r=1$,  the statement amounts to saying that the set of all Weyl $1$-planes of $X^n_s$ coincides with the effective Weyl orbit of a fixed line $L_\mathcal{I}$. This follows easily from  \cref{rmk weyl orbit curves}. 

 Assume now that $r\ge2$.
We define a map $\Phi_r$ from the effective Weyl orbit of $c_\mathcal{I}$ to the set of Weyl $r$-planes as follows:
$$
\Phi_r: w(c_\mathcal{I})\mapsto w(L_\mathcal{I}).
$$

First, we show that it is surjective. Take $W$ any Weyl $r$-plane. Then, by definition, there is an index set $\mathcal{J}$ with $|\mathcal{J}|=r+1$ and a Weyl transformation $w_\mathcal{J}\in W^n_s$ such that $w_\mathcal{J}(L_\mathcal{J})=W$.  
By  \cref{rmk weyl orbit curves}, there is $\omega'\in W^n_s$ such that $\omega'(c_\mathcal{I})=c_\mathcal{J}$. Then we have
$\Phi_r((w_\mathcal{J}\circ\omega')(c_\mathcal{I}))=W$.

 To prove injectivity, consider two distinct curves $w_1(c_\mathcal{I})$ and $w_2(c_\mathcal{I})$, for $w_1,w_2 \in W^n_s$ and assume, without loss of generality, that  $\deg w_2(c_\mathcal{I})\ge \deg w_1(c_\mathcal{I})$.
Applying the transformation $w_1^{-1}$ to both curves, we obtain the distinct curves $c_\mathcal{I}$ and $(w_1^{-1}\circ w_2)(c_\mathcal{I})$. 
It is enough to show that 
$(w_1^{-1}\circ w_2)(L_\mathcal{I})\ne L_\mathcal{I}$. 
{If $(w_1^{-1}\circ w_2)(c_\mathcal{I})$ is not effective, then it is easy to see that also $(w_1^{-1}\circ w_2)(L_\mathcal{I})$ is not effective and we are done.}
Now, if $\deg (w_1^{-1}\circ w_2)(c_\mathcal{I})=\deg c_\mathcal{I}=r$, then $(w_1^{-1}\circ w_2)(c_\mathcal{I})=c_\mathcal{J}$ for some $\mathcal{J}$ with $|\mathcal{J}|=r+1$ and $\mathcal{I}\ne \mathcal{J}$. Hence the two curves sweep out two distinct  Weyl $r$-plane: $L_\mathcal{I}$  and $L_\mathcal{J}$ respectively. 
On the other hand, if $(\deg w_1^{-1}\circ w_2)(c_\mathcal{I})>r$, then the $r$-dimensional subvariety it sweeps out will interpolate more than $r+1$ points, 
hence it cannot be a Weyl $r$-plane, proving in particular that $(w_1^{-1}\circ w_2)(L_\mathcal{I})\ne L_\mathcal{I}$.
\end{proof}

We are ready the give the proof of the main result of this section, which is ane easy consequence of the previous proposition.

\begin{proof}[Proof of  \cref{infinity}] 
Let $\mathcal{I}\subset\{1,\ldots,s\}$ be a subset of length $r+1$.
We know by   \cref{cor}  that the effective Weyl orbit of $c_\mathcal{I}$ is infinite.
 We conclude by  \cref{1-1}.
\end{proof}

\begin{remark}  Let $X^n_s$ be not a Mori dream space. We summarize the results obtained:
\begin{enumerate}
\item In $X^5_9$ there are finitely many $(-1)$-curves \cite{DM2} and infinitely many Weyl $r$-planes for $2\leq r\leq 4$. 
In particular there are finitely many $(-1)$-Weyl lines.

\item In $X^n_{n+4}$ there are finitely many $(-1)$-Weyl lines and infinitely many Weyl $r$-planes for $2\leq r\leq n-1$.
\item In $X^n_s$ with $s\geq n+5$ there are infinitely many $(-1)$-Weyl lines and infinitely many Weyl $r$-planes for $1\leq r\leq n-1$. 
\end{enumerate}
\end{remark}

\begin{remark}
It is clear that the elements of the orbits of $c_I$ for $I$ of length $r+1$ are among the generators of the cones $\mathcal C_{n-r}$.
\end{remark}

\begin{remark}
If $s\le 2^n$, by \cite{CLO} and \cite{DP-positivityI}, it is known that  the Mori cone of $X^n_s$ is finitely generated by the classes $e_i$ and $c_I$, for $I$ of length $2$, while all other $1$-Weyl lines are in $\mathcal{C}_{n-1}\setminus \mathcal{C}_{n-2}$ but 
are not extremal on $\mathcal{C}_{n-1}$. Moreover, we have seen that if $X^n_s$ is a MDS, then the latter set of curves is finite, while if $X$ is not a Mori dream space, the set is infinite. If $s>2^n$, we do not know if  $\mathcal{C}_{n-1}$ has finite or infinite extremal rays; in any case, any non-linear extremal ray is not spanned by $1$-Weyl lines.
\end{remark}

\subsection{Weyl stable base locus} \label{WSBL section}

\begin{lemma}\label{weyl base locus} 
Let $w$ be an element of the Weyl group of $X^n_s$. The multiplicity of containment of a Weyl $r$-plane $W=w(L_{I})$ in the base locus of an effective divisor $D$ is
\begin{equation}\label{Weyl-mult}
k_W(D)=\max( -w(c_I) \cdot D,0)
\end{equation}
where $c_I$ is defined as in \eqref{Cremona zero}.
\end{lemma}

\begin{proof}
  We first recall that if the  Weyl $r$-plane is linear $L_{I}$, the statement holds by the Base Locus Lemma \cite[Proposition 4.2]{dumpos}.

  Assume that $W$ is not a linear Weyl $r$-plane. The multiplicity of containment of the Weyl cycle $W=w(L_{I})$ in the base locus of an effective divisor $D$ is equal to the multiplicity of containment of the linear cycle $L_{I}$ in the base locus of the effective divisor $w^{-1}(D)$,
\begin{align*}
    k_{W}(D)&=k_{L_{I}}(w^{-1}(D))\\
  &=\max(-c_I\cdot w^{-1}(D),0)\\
  &=\max(-w(c_I)\cdot D,0)
  \end{align*}
  the last equality follows from the fact that intersection numbers are preserved by Cremona transformations.
\end{proof}

\begin{theorem} \label{Weyl-SBL}
The Weyl $r$-planes are stable base locus subvarieties of $X^n_s$.
\end{theorem}

\begin{proof}
Let $D$ be an effective divisor in $X^n_s$ and $W$ a Weyl $r$-plane, i.e. $W$ is an element of the effective Weyl orbit of a linear plane $L_{I}$.
Let $c=c_{I}$ be the curve class sweeping out the linear cycle $L_{I}$; then $w(c)$ sweeps out the Weyl cycle 
$W=w(L_{I})$.

Assume that the divisor $D$ contains $W$ in its base locus, so that $k_W(D)=-w(c)\cdot D>0.$ Consider an integer $m$ and notice that the Weyl cycle $W$ is in the  base locus of the divisor $mD$ because
\begin{align*} 
k_W(mD)&=-w(c)\cdot mD\\
&=-m\cdot(w(c)\cdot D)\\
&=m\cdot k_{W}(D)>0.
\end{align*}
\end{proof}

\begin{definition} \label{WSBL}
We define the \emph{Weyl (stable) base locus} of an effective divisor $D$ on $X^n_s$ to be the non-reduced subscheme of the stable base locus of $D$, such that any irreducible component is a Weyl $r$-plane. More explicitly, thanks to the previous results, the Weyl base locus of a divisor $D$ is:
$$
\bigcup_W k_W(D)W,
$$
where   $W$ is a Weyl $r$-plane and
$k_W(D)$ is the integer defined in \eqref{Weyl-mult}.
\end{definition}

We expect that the stable base locus and the Weyl base locus coincide for  $X^n_{n+4}$ and $X^n_{n+5}$, see Question \ref{quest}. For $X_{n+6}^n$ we know that the existence in the base locus of a quadric through nine points or a cone on it, for $n\ge 3$, is related to the speciality of linear systems (see e.g. \cref{ex-quadric}), hence we expect that the stable base locus may be different than the Weyl base locus.

\section{Cones of divisors, cones of curves and Weyl chamber decomposition for non-Mori dream spaces}
\label{section seven}
In this section we discuss our point of view on the birational geometry of non-Mori dream spaces of type $X^n_s$, that is based on the action of the Weyl group studied in Section \ref{section:infinite Weyl}. In particular,  after observing the remarkable role played by the anticanonical curve class in the description of the pseudoeffective cone of divisors, we propose a chamber decomposition of the latter, that is  governed by the action of the Weyl group of $X^n_s$. 

While observing that $F$ is a moving curve, or a limit of moving curves, for all Mori dream spaces $X^n_s$, we recall that the same holds for the non-Mori dream space $X^3_8$, following work of \cite{SX}, and we prove it for $X^5_9$.

We conclude the section by posing a few conjectures about chamber decompositions.

\subsection{The anticanonical curve class and the negative part of the pseudoeffective cone}
\label{section F}

The anticanonical curve class $F$, see Definition \ref{DEF anticanonical curve}, plays a fundamental role in  Coxeter theory, see \cite{DM3} for more details.

Recalling the Dolgachev-Mukai pairing \eqref{DM},  we say that a divisor class $D$ lies in the \emph{negative part} of the 
N\'eron-Severi space $N^1(X^n_s)_\R$ if $$\langle D,K_{X^n_s}\rangle\le0,$$ where $K_{X^n_s}$ denotes the canonical divisor class. 
Equivalently, we can say that a divisor lies in the negative part of $N^1(X^n_s)_\R$ 
if $D\cdot F\ge0,$
by using the intersection product \eqref{n+3: int pair}. For this reason we will denote by {$$F^\vee:=\{D: D\cdot F\ge 0\}\subset N^1(X^n_s)_\R$$}
the negative part of $N^1(X^n_s)_\R$.

The pseudoeffective cone is split in a negative part and a positive part. Like in the case of surfaces $X^2_s$, the negative part is easier to understand, while the positive part is wilder, see e.g.\ \cite{deFernex}.
Since $\mathcal{C}_0$, the closure of the cone of moving curves, is  dual to the pseudoeffective cone, we observe that the latter lies  in the negative part $F^\vee$ if and only if $F\in\mathcal{C}_0$, i.e. if $F$ is a limit of moving curves. 
Understanding for which values of $(n,s)$ the anticanonical curve class $F$ is moving, or a limit of moving curves, is therefore crucial. 

We start with an effectivity result.

\begin{proposition}\label{F-effective}
Let $X=X^n_s$ be the blown-up projective space at $s$ general points. Assume either $s\leq 2n+2$ and $n\geq 3$, or $s\leq 9$ and $n=2$. Then the anticanonical class $F$ is effective.  
\end{proposition}
\begin{proof} {It is enough to prove the statement for $X^2_9$ and for $X^n_{2n+2}$, for $n\ge3$.}

{In $X^2_9$,  $F$ is  the unique elliptic curve passing through nine points in general position.}

In $X=X^n_{2n+2}$, the  curve class $F$ can be written as a  nonnegative linear combination   {of lines through two points as follows:}
 $$F=\sum_{i=1}^{n+1}(h-e_{2i-1}-e_{2i}).$$
\end{proof}

Notice that, for $n=2$ and $s\ge10$, the anticanonical divisor $F$ is not effective.

In the next  result we prove that $F$ lies in $\mathcal{C}_0$ for every Mori dream space $X^n_s$, as well as in a few sporadic non-Mori dream cases, that occur for small values of $s$.

\begin{proposition}\label{mds}
Assume that either $X=X^n_s$ is a Mori dream space (i.e. it falls in the list \eqref{list MDS}), or  $X\in\{X^2_9,X^3_8,X^5_9\}$. If $D\in\Pic(X)$ is an effective divisor, then 
$$D\cdot F\geq 0.$$
\end{proposition}
\begin{proof}
Notice that if $D$ is a Weyl divisor on $X^n_s$, then $D\cdot F=1$. In fact, since the intersection pairing is preserved under the Weyl action and $F$ is Cremona invariant, 
then for every Weyl divisor $D$ there exists $w\in W^n_s$ with $w(D)=E_i$, for some $i$. Hence, we have $D\cdot F=E_i\cdot F=1$.
Since for all Mori dream spaces $X^n_s$,  Weyl divisors generate the effective cone, we conclude that $D\cdot F\ge 1$ for every effective divisor $D$.

Consider now $X^2_9$, $X^3_8$. 
The pseudoeffective cone in these cases  is generated by the Weyl divisors and by the anticanonical divisor class  \eqref{anticanonical divisor}. 
This is proved in \cite{nagata} (see also \cite{deFernex})  for $X^2_9$ and
in \cite[Theorem 6.5]{SX} for
$X^3_8$.
Since in both cases $-K_X\cdot F=0$, we conclude again that $D\cdot F\ge 0$ for every effective divisor.

Finally, we prove the statement for $X=X^5_9$.
In this case, we show that $D$ is effective and $D\cdot F=0$ if and only if $D$ lies in the ray generated by the anticanonical divisor $-K$. Accepting this claim, the effective cone of $X^5_9$ is contained in the halfspace $F^\vee=\{D\in \N^1(X)_\mathbb{R}| D\cdot F\ge 0\}$ or in the halfspace $\{D\in \N^1(X)_\mathbb{R}| D\cdot F\le 0\}$, and so does its closure, the pseudoeffective cone. Now, since there are effective divisors lying in $F^\vee$, e.g. the exceptional divisors, then we conclude that $\overline{\Eff}(X^5_9)\subset F^\vee$ and that the only intersection with $F^\perp$ consists of the ray spanned by the antincanonical divisor.

We now prove the claim. That $D=m(-K)$ is effective and has null intersection with $F$ is obvious, so we prove the converse statement.
Set $D=dH-\sum_{i=1}^9m_i E_i$; the condition $D\cdot F=0$ reads
\begin{equation}\label{DF=0}
6d=m_1+\cdots+m_9.
\end{equation}
Now, since the intersection number is preserved by the Weyl group action, and since the curve $F$ is Weyl invariant, we have $w(D)\cdot F=D\cdot F$, where $w\in W^5_9$, therefore we may assume that $D$ is Cremona reduced, i.e. that
\begin{equation}\label{cremonareduced}
    4d\ge \sum_{i\in I}m_i, \quad \forall I\subseteq\{1,\dots,9\}, |I|=6.
\end{equation}
Now, using \eqref{DF=0} and \eqref{cremonareduced}, we obtain
\begin{equation}\label{cremonacomplement}
2d\le \sum_{i\notin I}m_i,  \quad \forall I\subseteq\{1,\dots,9\}, |I|=6.
\end{equation}
Now, using \eqref{DF=0} and \eqref{cremonacomplement}, we obtain
$$
0=\sum_{i=1}^9m_i-6d=(m_1+m_2+m_3-2d)+(m_4+m_5+m_6-2d)+(m_7+m_8+m_9-2d)\ge0,
$$
from which, permuting indices, we obtain that $m_{i_1}+m_{i_2}+m_{1_3}=2d$, for every three distinct indices $i_1,i_2,i_3$, so $m_1=\cdots=m_9$. From this, we obtain that $6d=9m_1$ hence $D$ is a multiple of $3H-2\sum_{i=1}^9E_i$. Since $-K_X$ is Cremona invariant we conclude the proof.
\end{proof}

\begin{corollary}
Let $X$ as in Proposition \ref{mds}. Then the anticanonical curve $F$ is a limit of moving curves in $X$.
\end{corollary}
\begin{proof} 
From Proposition \ref{mds} we obtain that the effective cone lies completely in the negative part of $N^1(X)_\R$. Moreover, if $X\in\{X_9^2, X_8^3, X_9^5\}$, then $\Eff(X)$ shares just one ray (the anticanonical ray) with the hyperplane $F^\perp$. Taking the closure, we conclude that $\overline{\Eff}(X) \subset F^\vee$. This implies that $F$ is in $\mathcal C_0$, which is the cone dual to the pseudoeffective cone.  In particular, if $X$ is a MDS, then $F$ lies in the interior of $\mathcal C_0$, whereas if $X\in \{X_9^2, X_8^3, X_9^5\}$, then $F$ spans an extremal ray of $\mathcal C_0$.
\end{proof}

For higher values of $s$, the curve $F$ is not moving in general, as proved in the following result.
\begin{proposition}\label{prop n+6}
Assume $X=X^n_{n+6}$ and $n\ge3$ or $X=X^n_{n+5}$ and $n\ge4$. Then there are effective divisors lying  strictly in the positive part of the pseudoeffective cone of $X$. In particular the anticanonical curve $F$ is not  a limit of moving curve classes in $X$, i.e. $F\notin\mathcal{C}_0$.
\end{proposition}
\begin{proof}
We consider the case $X=X^n_{n+6}$ and $n\ge3$. Assume first $n=3$ and
consider the  quadric  of  $X^3_9$, given by $D=-\frac{1}{2}K_{X^3_9}=2H-\sum_{i=1}^9E_i$.  
For $n\ge 4$, by taking cones  over the quadric, we have effective divisors of the form $D=2H-\sum_{i=1}^9E_i-2\sum_{i=10}^{n+6}E_i$. In both cases 
 $D\cdot F=-1$, hence $D$ lies in the positive part of the pseudoeffective cone.

Now we consider the case 
$X=X^n_{n+5}$ and $n\ge4$.
The divisor $D=5H-3\sum_{i=1}^8E_i-2E_9$ is effective on $X^4_9$, because $h^0(D)= 1 $ and we have $D\cdot F=-1$. 
By taking cones over $D$, we obtain effective divisors that intersect negatively $F$ in $X=X^n_{n+5}$, for any $n\ge 5$.
\end{proof}

The only remaining {open} cases are the following.
\begin{question}\label{que74}
Assume that $X=X^n_{n+4}$, for $n\ge 6$. 
Is the anticanonical curve class $F$ on $X$ not  a limit of moving curves? 
\end{question}

\subsection{Weyl chamber decomposition}\label{WCD section}

In this section we propose  an application of the study of Weyl $r$-planes as stable base locus subvarieties of $X^n_s$.

Recalling  \eqref{Cremona zero} and 
using the notation of Lemma \ref{weyl base locus}, we consider the set of curve classes on $X^n_s$: 
$$\{w(c_I): w(c_I)\in (W^n_s\cdot c_I)^+, \; I\subset \{1,\dots,s\}, \; 2\le |I|\le n\}.$$
Recall also that each such curve sweeps out the Weyl $r$-plane $w(L_I)$.
We consider the (infinite) hyperplane arrangement in the N\'eron-Severi space given by 
\begin{equation}\label{Weyl-hyperplane-arrangement}
\{D\cdot w(c_I)=0: w(c_I)\in (W^n_s\cdot c_I)^+, I\subset \{1,\dots,s\}, 2\le |I|\le n\}.\end{equation}

By Lemma \ref{weyl base locus}, each such hyperplane splits the pseudoeffective cone of divisors in two regions, separating the divisors $D$ that contain $w(L_I)$ in their stable base locus, from those that do not.
This hyperplane arrangement induces a wall-and-chamber decomposition of the pseudoeffective cone of divisors (and of any cone $\mathcal{D}_k$), where each chamber has the property that divisors lying in the interior of the chamber have the same reduced Weyl stable base locus, cf. \cref{WSBL}. 
\begin{definition}
The \emph{Weyl chamber decomposition} is the wall-and-chamber decomposition of the pseudoeffective cone of divisors of $X^n_s$ induced by the hyperplane arrangement \eqref{Weyl-hyperplane-arrangement}. We say that two effective divisors are {\it Weyl equivalent} if they lie in the interior of the same Weyl chamber.
\end{definition}

Since the Weyl stable base locus scheme is contained in the stable base locus, the Weyl chamber decomposition ($\WCD$) can be refined to the {stable base locus decomposition} ($\SBLD$). 

We recall that for Mori dream spaces of type $X^n_s$, the WCD, the SBLD and the Mori chamber decomposition (MCD)  of the effective cone coincide, see e.g. \cite[Chapter 2]{BCP} and \cite[Theorem 5.7]{BDPS}, showing how the action of the Weyl group $W^n_s$  on curves governs completely the birational geometry of these spaces. In fact the following holds.
\begin{enumerate}
    \item  All extremal rays of the effective cone are Weyl $(n-1)$-planes, as proved by Mukai.
    \item All small $\mathbb{Q}$-factorial modifications are sequences of flips of (varieties birational to) Weyl $r$-planes, with $r<n-1$;  each such subvariety is swept out by a curve class in the orbit of $c_\mathcal I$ for some $\mathcal I\subset\{1,\ldots,s\}$ of length $r+1$, as proved in \cite{BDPS}.
    \item On the dual side, the extremal rays of all the cones of $r$-moving curves are either $0$-moving, $1$-moving, or $r$-moving in a Weyl $(n-r)$-plane, see Theorems \ref{strong-duality-thm-n+3},  \ref{strong-duality-thm-X^4_8} and \ref{strong-duality-thm-X^3_7}.
    \item The cone of $0$-moving curves in $X^n_s$, that is the dual of the effective cone of divisors, is given by $0$-Weyl lines and $1$-Weyl lines \cite{DM2}.
\end{enumerate}

We now turn our attention to the case of non-Mori dream spaces of type $X^n_s$. In this case even the description of the pseudoeffective cone of divisors seems out of reach for most cases (when $F$ is not in $\mathcal{C}_0$).
Even more, we do not have a definition of Mori equivalence as in \cite{Hu-Keel} because the section rings of $\Q$-Cartier divisors might not be finitely generated. In spite of this, we may ask whether an extension of the notion of Mori chamber decomposition of the pseudoeffective cone of $X^n_s$ (negative part) can be achieved, i.e. a \emph{nef chamber decomposition} where chambers would be pull-backs of  nef cones of infinitely many rational contractions
of $X^n_s$. More precisely,
a chamber is realised as $f^\ast \textrm{Nef}(Y)\ast \textrm{ex}(f)$, where $f:X\dasharrow Y$ is a finite sequence of flips of subvarieties of $X$ followed by a divisorial contraction.

The first open cases occur for {$X^3_8$ and $X^n_{n+4}$, for $n\ge 5$}. 

\subsection{Open questions on cones of divisors}
\label{Xn+4 section}

Recall that all divisors in the interior of a Weyl chamber share the same (reduced) Weyl stable base locus scheme, whose components form a poset, where the order is given by inclusion of varieties and the minimal elements are Weyl $1$-planes.
We assume that we may flip all the components, following the inclusion order and starting from the curves. In particular we are assuming that \cite[Conjecture 2.2]{DM2} holds, i.e. that all Weyl $1$-planes are disjoint in $X^n_s$ and, moreover, that the birational images of the Weyl $r$-planes in the poset are disjoint after having flipped the Weyl $(r-1)$-planes, for every $r\ge2$.

We first focus on the cases $X^5_9$ and $X^3_8$, where  the pseudoeffective cone has  only a negative part, see Proposition \ref{mds}.
We propose the following conjecture. 

\begin{conjecture}\label{WCD-X59}
Let $X=X^5_9$ or $X=X^3_8$.
Then
\begin{enumerate}
\item $
\WCD(\overline{\Eff}(X))= 
\SBLD(\overline{\Eff}(X)). 
$
\item Each Weyl chamber is of the form $f^\ast \textrm{Nef}(Y)\ast \textrm{ex}(f)$, where $f:X\dasharrow Y$ is a rational contraction.
\end{enumerate}
\end{conjecture}

\begin{remark}
Recall that the nef cones of $X^3_8$ and of $X^5_9$ are rational polyhedral, see e.g. \cite{CLO}and \cite{DP-positivityI}.
Since, by Theorem \ref{infinity}, there are infinitely many Weyl $2$-planes and divisors, then the Weyl chamber decompositions, in both cases, consists of infinitely many chambers. 
It would be interesting to know whether all such
chambers are polyhedral or if some of them are not.

Stenger and Xie in \cite[Proposition 1.1]{SX} proved an existence result for a subdivision of the effective movable cone of $X^3_8$ into nef chambers. It would be interesting to see if the latter can be described explicitly in terms of Weyl chambers and if their result can be extended to the whole pseudoeffective cone of  $X^3_8$; this would  provide a proof of  Conjecture \ref{WCD-X59} for $X^3_8$. To the best of our knowledge, the conjecture is completely open for $X^5_9$.
\end{remark}

Less is known for $X^n_{n+4}$, with $n\ge 5$. For instance, we do not know whether the pseudoeffective cone has a positive part for $n\ge6$, see \cref{que74}.
For $s\ge n+5$ and $n\ge3$, excluding $X^3_8$, the pseudo-effective cone certainly has a positive part, see Proposition \ref{prop n+6}.
For this reason, we pose the following question.

\begin{question} \label{question neg-eff nonmds}
Assume that $X=X^n_{s}$ is not a Mori dream space that $n\ge 3$, $s\ge n+4$ and that $X\notin\{X^3_8, X^5_9\}$.
\begin{enumerate}
\item Are all the extremal rays of the pseudoeffective cone of $X$ that lie in $F^\vee \setminus F^\perp$ Weyl divisors?
\item Do the statements of Conjecture \ref{WCD-X59} hold for the negative part of the pseudo-effective cone of $X$?
\end{enumerate}
\end{question}

\subsection{Open questions for cones of curves}
\label{FINAL}

Recall that in $X^n_{n+4}$, there are finitely many $(-1)$-Weyl curves, which are the lines through 2 points and the rational normal curves of degree $n$ through $n+3$ points, while in $X^n_{n+5}$ they are infinitely many.
In \cite{DM2}, the authors posed the question as to whether $X^n_{n+4}$, for $n\ge6$, and $X^n_{n+5}$, for $n\ge4$,  possess other $(-1)$-curves besides these $(-1)$-Weyl lines. 

Moreover, for non-Mori dream spaces, the Weyl orbits of $(0)$- and $(1)$-Weyl lines are infinite, see \cite[Corollary 5.8]{DM2}.
Recall that any effective divisor on $X^n_s$ intersects all $(0)$- and $(1)$-Weyl lines non-negatively, as proved in \cite[Theorem 6.7]{DM2}, see also Remark \ref{rmk35}. 

Recall that if $X^n_s$ is a Mori dream space,
the extremal rays of $\mathcal C_0$ 
are the $(0)$-Weyl lines and the $(1)$-Weyl lines, see \cref{rmk35}. 
Moreover Proposition \ref{mds}
implies that the curve $F$ spans an extremal ray of the cone of moving curves of $X^5_9$ and $X^3_8$.
Since the pseudoeffective cones of $X^5_9$ and $X^3_8$ have infinitely many extremal rays, including the Weyl divisors, then so do their duals. We propose the following conjecture.

\begin{conjecture} \label{conj-5-9}
Let $X$ be either $X^5_9$, or $X^3_8$. 
The cone $\mathcal{C}_0$ of moving curves of $X$ is generated by the following three Weyl orbits: the anticanonical curve class $F$, the $(0)$-Weyl lines and the $(1)$-Weyl lines.
\end{conjecture}

For higher dimension, we expect that a similar statement to that of the above conjecture is false in general.
Indeed, the pseudoeffective cone will have a positive part, since the anticanonical class $F$ is no longer moving. However, we ask if the negative part of the pseudoeffective cone continues to behave well.

\begin{question}
Let $X=X^n_s$, for any  $n,s$. A divisor in the negative part $F^\vee$ of the pseudoeffective cone of $X$ is effective if and only if it intersects non negatively all the $(0)$-Weyl lines and $(1)$-Weyl lines?
\end{question}

Finally, we propose the following open questions.
\begin{question}\label{quest} 
Let either $X=X^n_{n+4}$ and $n\ge5$, or $X=X^3_8$. 
\begin{enumerate}
\item Is it true that (1),(2),(3) of Theorem \ref{main} are equivalent for $X$?
\item 
Which are the extremal rays of cones $\mathcal{C}_k$ of $k$-moving curves in $X$?
\item Does the strong duality hold for $X$? Namely for $0\leq r\leq n-1$ do we have
$\mathcal{C}_r=\mathcal D_r^\vee?$
\end{enumerate}
\end{question}

\begin{remark}
Some implications of the statements of Question \eqref{quest} hold for arbitrary $s$ in non-Mori dream spaces $X^n_s$. In this direction, we recall 
the following.
\begin{enumerate}
\item The cone of moving curves $\mathcal{C}_0$  contains the two infinite Weyl orbits:  $(0)$-Weyl lines and $(1)$-Weyl lines  as generators.

\item The effective Weyl orbits of $h-e_i$, $e_j$ and 
$c_\mathcal{I}$ with $|\mathcal{I}|=n-k+1$
are among the 
generators of  the cone of $k$-moving curves $\mathcal{C}_k$. 

\item The inclusion $\mathcal{C}_k\subset \mathcal{D}_k^\vee$ always holds.
\end{enumerate}
\end{remark}

\section{Weyl expected dimension of linear systems}\label{section dimension}
In \cite[Definition 3.2]{BDP-TAMS}, the authors proposed  the notion of 
{\it linear expected dimension}, which takes into account the contribution of linear cycles in the stable base locus. They also proved that this is in fact the actual dimension for any linear system in $X^n_s$ with $s\le n+2$.
In \cite[Definition 6.1]{BraDumPos3},  the more general notion of {\it secant linear dimension} for linear systems in $X^n_s$ with $s\le n+3$ was proposed, and in \cite{LPSS}  it was proved to match the actual dimension.
In the cases $X^3_7$ and $X^4_8$, the definition of 
{\it  Weyl expected dimension} was given in \cite[Definition 2]{BDP-Ciro}.

Here we propose a new definition of {\it Weyl expected dimension}, that extends all the previous ones and  applies to  any $X_s^n$. 

\begin{definition}\label{weyl exp} 
Let $D$ be an effective divisor on $X=X^n_s$. We call  Weyl virtual (affine) dimension of the linear system $\mathcal{L}=\PP(\HH^0(X,\OO(D)))$ the number

\begin{equation}\label{wdim}\wdim(\LL)=\chi(X, \mathcal{O}(D))+\sum_{r=1}^{n-1}  \sum_{W\in \mathcal W(r)}(-1)^{r+1}\binom{n+k_W(D)-r-1}{n}, 
\end{equation} 
where $\mathcal W(r)$ is the set of all the Weyl $r$-planes and the multiplicity of containment $k_W(D)$ of the Weyl plane $W$  in the base locus of $D$ can be explicitly computed by Remark \ref{explanation}.

We define the Weyl expected (affine) dimension of $\LL$ as 
 the maximum between  $\wdim(\LL)$ and $0$.
\end{definition}

\begin{remark}\label{explanation}
We recall that, given $r$, a Weyl $r$-plane $W\in \mathcal W(r)$ is an element of the effective orbit $(W_s^n\cdot L_I)^+$ for an index set $I$ of length $r+1$, and, by Proposition \ref{1-1}, it corresponds to an element of the effective orbit $(W_s^n\cdot c_I)^+$ of the  curve class $c_I$, \eqref{Cremona zero}.
By Lemma \ref{weyl base locus}, the multiplicity of containment $k_W(D)$ of a Weyl plane $W=w(L_I)$  in the base locus of  a divisor $D$ is obtained intersecting $D$  with the unique curve sweeping it out, $w(c_I)$, 
$$k_W(D)=\max( -w(c_I) \cdot D,0).$$ 
\end{remark}

\begin{remark}
 We observe that the sum \eqref{wdim} is finite for any effective divisor $D$. In fact, 
for every $I$,  even when the orbit of the Weyl $r$-plane $L_I$ is infinite, 
only finitely many elements of the latter, $W\in (W_s^n\cdot L_I)^+$, can be contained in the base locus of $D$, 
because the number of components of the base locus is finite, by Lasker-Noether Theorem.
 Consequently, since $k_W(D)$ is the multiplicity of containment of $W$ in the base locus of $D$,   
 only finitely many Newton binomials in \eqref{wdim} are nonzero.
\end{remark}
\begin{remark}\label{rem}
We observe that for $s=n+2$ the Weyl expected dimension of a non-empty linear system equals its linear expected dimension defined in \cite{BDP-TAMS}, $\wdim(\LL)=ldim(\LL)$.
For $s=n+3$, the Weyl expected dimension of a linear system equals the secant linear expected dimension, $\wdim(\LL)=\sigma\mathrm{ldim}(\LL)$, defined in \cite{BraDumPos3}.
For $n=4$, $s=8$, this notion equals to the Weyl expected dimension for divisors in $X^4_8$ defined in \cite{BDP-Ciro}.
We conclude that Definition \ref{weyl exp} generalizes the notions of  dimensions of linear systems expected to hold for all Mori Dream Spaces of the form $X^n_s$.
\end{remark}

\begin{proposition}  The Weyl virtual dimension is preserved under the Weyl group action.
\end{proposition}
\begin{proof}
Let $D$ be an effective divisor on $X$ and $\mathcal{L}=\PP(\HH^0(X,\OO(D)))$ the associated linear system.
    Let $\Cr=\Cr_\Gamma$ be a standard Cremona transformation and we denote by $\Cr(\LL)=\PP(\HH^0(X,\OO(\Cr(D))))$. We need to prove that 
    $$\wdim(\mathcal L)=\wdim( \Cr(\LL))$$
We first consider the case $s=n+4$ and $D\in \Pic(X^n_{n+4})$ an arbitrary effective divisor in $X^n_{n+4}$.
It is well known that the dimension of a linear system is invariant under Weyl  group actions.

Let $G\in \Pic(X^n_{n+3})$ be a divisor obtained from $D$ by forgetting one point $p$ that is not in $\Gamma$; obviously
$G$ is an effective divisor on $X^n_{n+3}$. {Moreover, also $\Cr(\LL_G)$ is an effective divisor in $X^n_{n+3}$.}
By \cite{LPSS} we know that the dimension of any non-empty linear system in $X^n_{n+3}$ is equal to  the {\it secant linear dimension}, which coincides with the virtual Weyl dimension. Therefore we obtain
$$\wdim(\LL_G)= \dim(\LL_G)=\dim(\Cr(\LL_G))=\wdim (\Cr(\LL_G)).$$
We remark that the Weyl $r$-planes that appear in the base locus of $G$ only involve the Weyl $r$-planes based exclusively at the points away from $p$; in particular the $\wdim$ formula carries all fixed linear subspaces through base points of $\Gamma$.

We claim that, for all the remaining Weyl $r$-planes $W$ interpolating the point $p$, which is away from the indeterminacy locus, we have 

$$k_W(D)=k_{\Cr(W)}(\Cr(D)).$$

Remark first that the birational transformation $\Cr$ does not contract the Weyl $r$-planes $W$. Indeed, the transformation $\Cr$ only contracts linear spaces spanned by the points of $\Gamma$, and $p\notin \Gamma$. Moreover, we have seen that each Weyl $r$-plane $W$, is spanned by exactly one curve class say $c$. In particular $c$ interpolates the point $p$ so the curve class $c$ is not contracted by the birational map $\Cr$.

We further remark that the standard intersection product between a curve and a divisor is invariant under the Weyl group action \cite{DM2}; one concludes

$$k_W(D)=\max(-c\cdot D,0)
=\max(-\Cr(c)\cdot \Cr(D),0)
=k_{\Cr(W)}(\Cr(D)).
$$
To conclude for $s\geq n+5$ we proceed by induction on the number of points starting from $s=n+4.$
\end{proof}

In the next theorem we list a  collection of divisors known to have the Weyl expected dimension.
\begin{theorem}\label{results} Assume that one of the following conditions holds:
\begin{enumerate}
\item  $D \in \Pic(X^{n}_{s})$ is a divisor such that $D\cdot F\geq D\cdot h$.
\item  For $n\geq 4$, $D \in \Pic(X^{n}_{n+5})$ is a
cone over an effective divisor $D'\in \Pic(X^3_{8})$. 
\item For  $n\geq 4$, $D \in \Pic(X^{n}_{n+6})$ is a
cone over an effective divisor $D'\in \Pic(X^3_{9})$ and multiplicities bounded by $8$.
\item For $n\geq 3$, $D\in \Pic(X^{n}_{n+7})$ is a
cone over an effective divisor $D'\in \Pic(X^2_{9})$.
\item For $n\geq 3$, $D \in \Pic(X^{n}_{s+n-2})$ is a cone with vertex set of size $n-2=s-k^2$ over an effective divisor $D'\in \Pic(X^2_{k^2})$ with homogeneous multiplicities.
\item For $n\geq 3$, $D \in \Pic(X^{n}_{s+n-2})$ is a cone over an effective divisor $D'\in \Pic(X^2_{s})$ with multiplicities bounded by $11$.
\item For $n\geq 3$, $D \in \Pic(X^{n}_{s+n-2})$ is a cone over an effective divisor $D'\in \Pic(X^2_{s})$ with homogeneous multiplicities bounded by $42$.
\end{enumerate}
Then the dimension of the linear system associated to $D$ coincides with its Weyl expected dimension.    
\end{theorem}
\begin{proof}
The first statement follows from Theorem 5.3 of \cite{BDP-TAMS} and Remark \ref{rem}.

We recall that any effective divisor in $\Pic(X^2_{9})$ and $\Pic(X^3_{8})$ has the Weyl expected dimension by \cite{Castelnuovo} and \cite{DL}, respectively. 
By \cite[Theorem 5.8]{BDP-nine} this is true also for 
any effective divisor in $\Pic(X^3_{9})$ with multiplicities bounded by $8$.

Any effective divisor in $\Pic(X^2_{k^2})$ and all equal multiplicities has the Weyl expected dimension by \cite{CM}. This result also holds for effective divisors in $\Pic(X^2_s)$ with arbitrary multiplicities bounded by $11$ by \cite{DJ},
 or with homogeneous multiplicities bounded by $42$ \cite{dumnicki2}. 
 
 In all these cases, arguing as in the proof of \cite[Lemma 6.7]{BraDumPos3} it is easy to prove that $\wdim$ is preserved by cone reduction.
Hence
the result follows by taking cones over $D'$. 
\end{proof}

Notice that we give the definition of Weyl expected dimension in $X^n_s$ for any $s$, but if $s\ge n+6$, then there exist linear systems of dimension different from the Weyl expected dimension, as shown in the following example. 

\begin{example} \label{ex-quadric}
In $X^3_9$ consider the divisor $D=4H-\sum_{i=1}^9 2E_i$. Then the associated linear system has negative Weyl expected dimension but it is not empty because it contains the double quadric through the nine points.
\end{example}

Hence we pose the following question.
\begin{question}\label{conj-wdim}
For $s= n+4$, does the dimension of any linear system in $X^n_s$ coincide with its Weyl expected dimension?
\end{question}

The answer to Question \ref{conj-wdim} is affermative  for $s=n+2$ by \cite{BDP-TAMS} and for $s=n+3$ by \cite{LPSS} (see also \cite{BraDumPos3}) and Remark \ref{rem}, while for $n=4$ it coincides with Questions in Section 6 of \cite{BDP-Ciro}. Theorem \ref{results} presents a list of non Mori Dream Spaces for which this question has an affermative answer.

We finally recall that in the planar case, $X^2_s$, Question \ref{conj-wdim} is known as the Segre-Harbourne-Gimigliano-Hirschowitz conjecture. Even if important progress has been made in some particular cases, partially exposed in Theorem \ref{results}, this conjecture is still open in full generality, see  \cite{BP} for a short review.

\end{document}